\documentclass[12p]{amsart}
\usepackage{amscd,amsmath,amssymb,amsthm,amsfonts,epsfig,graphics}
\usepackage{geometry}
\usepackage{graphicx}
\usepackage{mathrsfs,amssymb}
\usepackage{bm}
\usepackage{subfigure}

\theoremstyle{plain}

\newtheorem{corollary}{Corollary}

\newtheorem{definition}{Definition}

\newtheorem{lemma}{Lemma}

\newtheorem{proposition}{Proposition}
\newtheorem{remark}{Remark}

\newtheorem{theorem}{Theorem}
\numberwithin{equation}{section}
\newcommand{\al}{\alpha}
\newcommand{\lap}{\Delta_g}
\newcommand{\eps}{\epsilon}

\newcommand{\lapp}{(-\Delta_g)^{\alpha/2} }
\newcommand{\ls}{\lesssim}
\newcommand{\gs}{\gtrsim}
\allowdisplaybreaks[4]
\begin{document}
\title[Fractional Schr\"odinger]
 {$L^p$ Eigenfunction bounds for fractional Schr\"odinger operators on manifolds}

\author{Xiaoqi Huang, Yannick Sire, and Cheng Zhang}


\address{Department of Mathematics, Johns Hopkins University, Baltimore, MD 21218, United States}
\email{xhuang49@jhu.edu; sire@jhu.edu}
\address{Department of Mathematics, University of Rochester, Rochester, NY 14627, United States}
\email{czhang77@ur.rochester.edu}


\keywords{}

\dedicatory{}

\begin{abstract}
This paper is dedicated to $L^p$ bounds on eigenfunctions of a Sch\"odinger-type operator $(-\Delta_g)^{\alpha/2} +V$ on closed Riemannian manifolds for critically singular potentials $V$. The operator $(-\Delta_g)^{\alpha/2}$ is defined spectrally in terms of the eigenfunctions of $-\Delta_g$. We obtain also quasimodes and spectral clusters estimates. As an application, we derive Strichartz estimates for the fractional wave equation $(\partial_t^2+(-\Delta_g)^{\alpha/2}+V)u=0$.  The wave kernel techniques recently developed by Bourgain-Shao-Sogge-Yao \cite{BSSY} and Shao-Yao \cite{sy} play a key role in this paper. We construct a new reproducing operator with several local operators and some good error terms. Moreover, we shall prove that these local operators satisfy certain variable coefficient versions of the ``uniform Sobolev estimates'' by Kenig-Ruiz-Sogge \cite{KRS}. These enable us to handle the critically singular potentials $V$ and prove the quasimode estimates.

\end{abstract}

\maketitle

\tableofcontents

\section{Introduction}

Quasimodes and eigenfunctions estimates for a Schr\"odinger operator have been considered by Blair, Sogge and the second author in \cite{BSS19}. The aim of the present article is to investigate these estimates when one considers a fractional Schr\"odinger operator of the form  $(-\Delta_g)^{\alpha/2} +V$ on compact Riemannian manifolds $(M,g)$ of dimension $n\ge2$. Here $0<\alpha<2$ is the L\'evy index, and the operator $(-\Delta_g)^{\alpha/2}$ stands for the so-called {\sl spectral fractional Laplacian} classically defined by functional calculus. We shall deal with real-valued  potentials $V(x)$ with critical singularities. The fractional Schr\"odinger operator arises in equations related relativistic models in stellar collapse (see e.g. \cite{LiebYau1,LiebYau2,daubechies,carmona,FLS1,FLS2}). See also \cite{Laskin,Laskin2,Laskin3}  for further discussions related to the fractional quantum mechanics.

In \cite{BSS19}, the authors consider the operator
$$H_V=-\Delta_g+V$$
on a closed manifold $(M,g)$. An instance of the results in that paper is the following theorem.

\begin{theorem} Assume that $n\ge 4$ and $V\in L^{\frac{n}2}(M)$ and let
\begin{equation}\label{a.3}
\sigma(p)=\max\big( \, n(\tfrac12-\tfrac1p)-\tfrac12, \, \tfrac{n-1}2(\tfrac12-\tfrac1p)\, \big).
\end{equation}
Then for $\lambda\ge1$ we have
\begin{eqnarray}\label{a.4}
\|u\|_{L^p(M)}\le C_{p,V}\bigl(\, \lambda^{\sigma(p)-1}\bigl\| \bigr(-\Delta_g+V-(\lambda+i)^2\bigl)u\Big\|_{L^2(M)}
+\lambda^{\sigma(p)}\|u\|_{L^2(M)} \, \Big ),
\\ \text{if } \, \, \, u\in C^\infty(M),
\end{eqnarray}
provided that
\begin{equation}\label{a.5}
2<p<\tfrac{2n}{n-3}.
\end{equation}
The constant $C_{p,V}$ depends on $p$, $V$ and $(M,g)$ but not on $\lambda$.
\end{theorem}

The integrability assumption $V \in L^{n/2}(M)$ reflects the criticality of the equation with respect to scaling. In lower dimensions, an additional assumption on the potential $V$, namely belonging to a Kato class, is required. Notice that the previous theorem can be proved using standard methods if the potential is assumed to be smooth.

The aim of the present paper is to contribute to the understanding of such eigenfunction estimates in the case of the spectral fractional Laplacian perturbed by a potential $V$. Notice that in the case $V \equiv 0$, by the spectral theorem, the eigenfunctions of $(-\Delta_g)^{\alpha/2}$ are the same as the ones of the Laplacian, and then the eigenfunction estimates can be found in the seminal paper \cite{sogge88}. Our approach in the current work builds on the one in \cite{BSS19} using uniform resolvent estimates. However, in the present setting, the situation is more involved since the expression of the resolvent of our operator is much more complicated  and requires deeper care. Indeed, it seems difficult to construct the parametrix for the ``non-local'' operator $(-\Delta_g)^{\alpha/2}-z$, and the standard parametrix construction in \cite{BSS19} or \cite{FKS} cannot work for these operators. To deal with this, we make use of an identity between the resolvent operators $((-\Delta_g)^{\alpha/2}-z)^{-1}$ and $(-\Delta_g-z)^{-1}$, and then exploit the wave kernel techniques recently developed by Bourgain-Shao-Sogge-Yao \cite{BSSY} and Shao-Yao \cite{sy}. Then we construct a new reproducing operator with several ``local'' operators and some good error terms. Moreover, we shall prove that these ``local'' operators satisfy certain variable coefficient versions of the ``uniform Sobolev estimates'' by Kenig-Ruiz-Sogge \cite{KRS}. These enable us to handle the critically singular potentials $V$ and prove the quasimode estimates.

We now state our main results. We first recall the suitable notion of Kato class in our setting.
\begin{definition}
  Let $n\ge2$ and $0<\alpha<2$. The potential $V$ is said to be in the Kato class $\mathcal{K}_\alpha(M)$ if
  \begin{equation}\label{kato}
    \lim_{r\downarrow 0}\sup_{x\in M}\int_{B_r(x)}d_g(x,y)^{\alpha-n}|V(y)|dy=0
  \end{equation}
  where $d_g(\cdot,\cdot)$ denotes geodesic distance and $B_r(x)$ is the geodesic ball of radius $r$ about $x$ and $dy$ denotes the volume element on $(M,g)$.
\end{definition}
Note that since $M$ is compact we automatically have that $V\in L^1(M)$ if $V\in \mathcal{K}_\alpha(M)$, and an easy argument shows that if $ V\in L^{\frac n\alpha+\varepsilon}(M), \hspace{1mm} \varepsilon>0$, then $V\in \mathcal{K}_\alpha(M)$.
\begin{theorem}\label{quasi1}
Assume that $n\ge 4$, $\frac{2n}{n+1}<\alpha<2$, $2\le p<\frac{2n}{n+1-2\alpha}$,  $V \in L^{\frac n\alpha}(M)$ and let
\begin{equation}\label{sigma}
\sigma(p)= \max\big(\, n(\tfrac 12-\tfrac1p)-\tfrac12, \tfrac{n-1}{2}(\tfrac12-\tfrac1p)\, \big).
\end{equation}
Then for $\lambda\ge 1$  we have
\begin{equation}\label{quasi}
\begin{aligned}
\|u\|_{L^{p}(M)}\le  C_V(\lambda^{1-\alpha+\sigma(p)}\|(\lapp+V-(\lambda+i)^\alpha)u\|_{L^2(M)} +\lambda^{\sigma(p)}\|u\|_{L^2(M)}) \qquad \\
{\rm if} \hspace{2mm} u\in C^\infty (M).
\end{aligned}
\end{equation}
The constant  $C_V$  depends on $p$, $V$, $\alpha$ and  $(M,g)$ but not on $\lambda$.
\end{theorem}

\begin{theorem}\label{quasi2}
Assume  that $\frac{2n}{n+1}<\alpha<2$ and $V \in L^{\frac {n}{\alpha}}(M)\cap \mathcal{K}_\alpha(M)$.  Then if $n=2$ or $3$, $\sigma(p)$ as in \eqref{sigma} and $\lambda \ge 1$, we have for $2\le p\le \infty$
\begin{equation}\label{quassi}
\begin{aligned}
\|u\|_{L^{p}(M)}\leq  C_V\lambda^{1-\alpha+\sigma(p)}\|(\lapp+V-(\lambda+i)^\alpha)u\|_{L^2(M)}   \\
{\rm if} \hspace{2mm} u\in {\rm Dom}(H_V).
\end{aligned}
\end{equation}
If $n\ge4$, this inequality holds for all  $2\le p\le \frac{2n}{n-2\alpha}$, and we also have for $\frac{2n}{n-2\alpha}< p\le \infty$,

\begin{equation}\label{quasssi}
\begin{aligned}
\|u\|_{L^{p}(M)}\leq  C_V\Big(\lambda^{1-\alpha+\sigma(p)}\|(\lapp+V-(\lambda+i)^\alpha)u\|_{L^2(M)}\\
+\lambda^{-N+n(\frac12-\frac1p)} \|(I+H_V)^{N/\alpha}R_\lambda u\|_{L^2(M)} \Big) \\
{\rm if} \hspace{2mm} u\in {\rm Dom}(H_V).  \hspace{0.37in}
\end{aligned}
\end{equation}
Assuming that $N>n/2$ with $R_\lambda$  being the projection operator for $P_V=H_V^{1/\alpha}$ corresponding to the interval $[2\lambda, \infty)$.
\end{theorem}
By modifying the examples in \cite[section 6]{BSS19}, it is not hard to see that when $V\equiv0$, \eqref{quassi} does not hold for $p>\frac{2n}{n-2\alpha}$ if $n\ge4$.  If $\chi_\lambda^V$ is the spectral projection operator associated with $P_V$ corresponding to the unit intervals $[\lambda,\lambda+1]$, then we have the following corollary.

\begin{corollary}\label{coro1}
  Let $n\ge2$, $\frac{2n}{n+1}<\alpha<2$, and $\sigma(p)$ as in \eqref{sigma}. If $V\in L^{n/\alpha}(M)\cap\mathcal{K}_\alpha(M)$, then
  \begin{equation}
    \|\chi_\lambda^Vf\|_{L^p(M)}\le C_V(1+\lambda)^{\sigma(p)}\|f\|_{L^2(M)},\ \ p\ge2,\ \lambda\ge0.
  \end{equation}
  Consequently, if $(\lapp+V)e_\lambda=\lambda^\alpha e_\lambda$ in the sense of distributions, we have
  \begin{equation}
    \|e_\lambda\|_{L^p(M)}\le C_V(1+\lambda)^{\sigma(p)}\|e_\lambda\|_{L^2(M)},\ \ p\ge2,\ \lambda\ge0.
  \end{equation}
\end{corollary}

\begin{remark}\label{rem1}{\rm
  The range of $\alpha$ in Theorem \ref{quasi2} can be improved if the potential $V$ has better regularity, by slightly modifying the proof. For example, if $V\equiv0$, then Theorem \ref{quasi2} holds for $\frac{n}{n+1}< \alpha<2$ if $n\ge4$, and $\frac{n}2<\alpha<2$ if $n=2,3$. Moreover, when $n=2,3$,  \eqref{quassi} also holds for $p\in[p_c,\frac{2n}{n-2\alpha}]$ if $\frac{n}{n+1}<\alpha< \frac{n}2$, and $p\in [p_c,\infty)$ if $\alpha=\frac{n}2$. Meanwhile, when $n=2,3$, \eqref{quasssi} also holds for $p\in(\frac{2n}{n-2\alpha}, \infty]$ if $\frac{n}{n+1}<\alpha< \frac{n}2$, and $p=\infty$ if $\alpha=\frac{n}2$. Clearly, Corollary \ref{coro1} is trivial if $V\equiv0$, since it holds for $0<\alpha<2$ by the definition of $(-\Delta_g)^{\alpha/2}$.

  If $V\in L^\infty(M)$, then Theorem \ref{quasi2} holds for $1\le \alpha<2$ if $n\ge4$, and  $\frac{n}2<\alpha<2$ if $n=2,3$. Furthermore, when $n=2,3$, \eqref{quassi} also holds for $p\in[p_c,\frac{2n}{n-2\alpha}]$ if $1\le\alpha< \frac{n}2$, and $p\in [p_c,\infty)$ if $\alpha=\frac{n}2$. Meanwhile, when $n=2,3$, \eqref{quasssi} also holds for $p\in(\frac{2n}{n-2\alpha}, \infty]$ if $\frac{n}{n+1}<\alpha< \frac{n}2$, and $p=\infty$ if $\alpha=\frac{n}2$. Consequently, Corollary \ref{coro1} holds for $1\le \alpha<2$ if $V\in L^\infty(M)$. }
\end{remark}

The paper is organized as follows. In the next section, we shall go over background concerning heat kernels and Kato class, and then introduce the construction of the reproducing operator that we shall use in proving our quasimode estimates. In Section 3, we shall establish some $L^q-L^p$ bounds for the operators coming from the wave kernel method. In Section 4, we shall prove Theorem \ref{quasi1}. In Section 5 and 6, we shall prove Theorem \ref{quasi2} for $n\ge3$ and $n=2$, respectively. In Section 7, as an application, we shall prove Strichartz estimates for wave operators involving potentials $V\in L^{n/\alpha}(M)\cap\mathcal{K}_\alpha(M)$.

Throughout this paper, $X\ls Y$ (or $X\gs Y$) means $X\le C_V Y$ (or $X\ge C_VY$) for some positive constant $C_V$ dependent on $V$. $X\approx Y$ means $X\ls Y$ and $X\gs Y$. In the following sections, the positive number $\delta$ depends on $V$ (see \eqref{d2}), hence the positive constant $C_\delta$ depends on $V$. Other positive constants like $C$, $C_0$, $C_1$, $C_2$, $C_M$ are independent of $V$. We just need to prove the theorems for sufficiently large $\lambda\ge C_\delta$, since relatively small $\lambda$ can be handled easily.

\section{Preliminaries results}

\subsection{Heat kernel estimates}
\begin{proposition}\label{bounds_heat_kernel}
Let $(M,g)$ be a closed Riemannian manifold of dimension $n\ge2$. Let $0<\alpha<2$ and $V\in \mathcal{K}_\alpha(M)$. Let $p^V(t,x,y)$ be the heat kernel of $H_V=(-\Delta_g)^{\alpha/2}+V$. Then for $0<t\le1$
\begin{equation}\label{heat1}
p^V(t,x,y)\approx  q_\alpha(t,x,y),\ \ x,y\in M
\end{equation}
where $q_\alpha(t,x,y)=\min\{t^{-n/\alpha}, td_g(x,y)^{-n-\alpha}\}$.
\end{proposition}
\begin{proof} In the Euclidean case, the claim has been proved in \cite{song2}, \cite{cks}. The manifold case can be proved along the same line of arguments used in \cite{song2} (cf. also  \cite{chensong}, \cite{cks}, \cite{bj07}, \cite{wz15}).
\end{proof}
The heat kernel estimates can be extended to all $t>0$ inductively by the semigroup property. Indeed, there are positive constants $\eps_1,\ \eps_2$ (dependent on $V$) such that for all $t>0$
\[e^{-\eps_1t} q_\alpha(t,x,y)\ls p^V(t,x,y)\ls e^{\eps_2t} q_\alpha(t,x,y),\ \ x,y\in M.\]In particular, when $V\equiv0$, the heat kernel estimate \eqref{heat1} holds for all $t>0$, which directly follows from subordination as in \cite[Theorem 4.2]{GG} (cf. also \cite{LY}, \cite{sturm}, \cite{sal}, \cite{BSS}). Moreover, $-H_V$ is the $L^2$-generator of the Schr\"odinger semigroup $e^{-tH_V}$ on $L^2(M)$ (see \cite[Corollary 4.9]{chensong}).

\subsection{Kato class and self adjointness}
We review that the assumption $V\in \mathcal{K}_\alpha(M)$ implies that the symmetric operators $H_V=(-\Delta_g)^{\alpha/2}+V$ are self-adjoint and bounded from below. The following is  the analog of  \cite[Proposition 2.1]{BSS19}.
\begin{proposition}\label{self}
If $V\in \mathcal{K}_\alpha(M)$ the quadratic form
\[q_V(u,v)=\int_MVu\bar{v}dx+\int\lapp u\bar{v}dx, \qquad u,v \in {\rm Dom}\big((\lapp+1)^{1/2}\big)\]
is bounded from below and defines a unique semi-bounded self-adjoint operator $H_V$ on $L^2$. Moreover, $C^\infty(M)$ constitutes a form core for $q_V$.
\end{proposition}
\begin{proof}
Since $\lapp $ is self adjoint, by the KLMN Theorem, it suffices to prove that for any $0<\epsilon<1$ there is a constant $C_\epsilon<\infty$ so that
\begin{equation}\label{below}
\int |V||u|^2dx\le \epsilon^2\|(-\Delta_g)^{\alpha/4} u\|_2^2+C_\epsilon\|u\|_2^2, \quad u\in {\rm Dom}\big((\lapp+1)^{1/2}\big)
\end{equation}
First, the heat kernel of $\lapp$ (e.g. \cite[Theorem 4.2]{GG}) satisfies
\[p(t,x,y)\approx q_\alpha(t,x,y),\ t>0,\ x,y,\in M\]
where $q_\alpha(t,x,y)=\min\{t^{-n/\alpha}, td_g(x,y)^{-n-\alpha}\}$. By the definition of $V\in \mathcal{K}_\alpha(M)$, it is not difficult to see
$$\sup\limits_{x\in M} \int_0^\infty\int_M e^{-Nt}p(t,x,y)|V(y)|dydt \rightarrow 0, \quad\text{as} \quad N\rightarrow \infty.
$$
Indeed, when $d_g(x,y)<N^{-\frac1{2\alpha}}$, \[\int_0^\infty e^{-Nt}q_\alpha(t,x,y)dt\ls d_g(x,y)^{\alpha-n},\]
and when $d_g(x,y)\ge  N^{-\frac{1}{2\alpha}}$,
\[\int_0^\infty e^{-Nt}q_\alpha(t,x,y)dt\ls N^{-2}d_g(x,y)^{-n-\alpha}+N^{\frac{n}\alpha-1}e^{-Nd_g(x,y)^\alpha}\ls N^{-1}d_g(x,y)^{\alpha-n}+N^{\frac{n}\alpha-1}e^{-\sqrt N}.\]
Similarly, we have
$$\sup\limits_{y\in M} \int_0^\infty\int_M e^{-Nt}p(t,x,y)|V(x)|dxdt \rightarrow 0, \quad\text{as} \quad N\rightarrow \infty
$$
Let $H_0=\lapp+1$. Choose $N=N_\epsilon$ so that the left hand side is $<\epsilon^2$, i.e,
$$\|(H_0+N_\epsilon)^{-1}|V|\|_{\infty}<\epsilon^2
$$
$$\||V|(H_0+N_\epsilon)^{-1}\|_{\infty}<\epsilon^2.
$$
If the operator $T:=|V|^{1/2}(H_0+N_\epsilon)^{-1}|V|^{1/2}$  has kernel $K(x,y)$, from the above two inequalities we know
$$\int_MK(x,y)|V|^{1/2}(y)dy \le \epsilon^2 |V|^{1/2}(x)
$$
and
$$\int_MK(x,y)|V|^{1/2}(x)dx \le \epsilon^2 |V|^{1/2}(y).
$$
So by Schur's test, we have
$$ \|T\|_{L^2\rightarrow L^2}= \||V|^{1/2}(H_0+N_\epsilon)^{-1}|V|^{1/2}\|_{L^2\rightarrow L^2} <\epsilon^2
$$
which, by a $TT^*$ argument, is equivalent to
\begin{equation}\label{self norm}
\||V|^{1/2}(H_0+N_\epsilon)^{-1/2}\|_{L^2\rightarrow L^2} <\epsilon
\end{equation}
This implies \eqref{below}.
\end{proof}

If $u \in \text{Dom}\big((\lapp+1)^{1/2}\big)$ then $\lapp u$ and $Vu$ are both distributions. If $H_V$ is the self-adjoint operator given by the proposition, then Dom$(H_V)$ is all such $u$ for which $\lapp u+Vu \in L^2$. At times, we abuse notation a bit by writting $H_V$ as $\lapp +V$.

Now if we take  $\epsilon^2=\frac12$ in \eqref{below}, we can get that for large $N$:
\begin{equation}\label{discrete}
\begin{aligned}
\|\big(\lapp+1\big)^{1/2}u\|_2^2=\int\big( \lapp+1\big)u\hspace{0.5mm}\bar{u}\hspace{0.5mm}dy&\le 2 \int\big( \lapp+V+N\big)u\hspace{0.5mm}\bar{u}\hspace{0.5mm}dy\\
&\le \|\sqrt{H_V+N}u\|_2^2
\end{aligned}
\end{equation}
Thus, $\big((-\Delta_g)^{\alpha/4}+1\big)(H_V+N)^{-1/2}$ is bounded on $L^2$. Since $\big((-\Delta_g)^{\alpha/4}+1\big)^{-1}$ is a compact operator, so must be $(H_V+N)^{-1/2}$. From this we conclude that the self-adjoint operator $H_V$ has $discrete \hspace{1mm}spectrum$.

\subsection{Construction of a reproducing operator}
We shall exploit the wave kernel techniques developed by Bourgain-Shao-Sogge-Yao \cite{BSSY} and Shao-Yao \cite{sy}. If $P=\sqrt{-\lap} $, we may split the resolvent operator for the Laplacian into local and nonlocal parts:
\begin{equation}\label{regular resolvent}
    \begin{aligned}
        (-\lap-(\lambda+i\mu )^2)^{-1}&=\frac{{\rm sgn} \hspace{1mm} \mu}{i(\lambda+i\mu)}\int_0^\infty\rho(t/\delta)e^{i({\rm sgn}\mu)\lambda t-|\mu|t}\cos tP dt\\
        &\quad+\frac{{\rm sgn} \hspace{1mm} \mu}{i(\lambda+i\mu)}\int_0^\infty(1-\rho(t/\delta))e^{i({\rm sgn}\mu)\lambda t-|\mu|t}\cos tP dt\\
        &:=T_{\lambda,\mu}+R_{\lambda,\mu}.
    \end{aligned}
    \nonumber
\end{equation}
Here $0<\delta<1$ is a fixed small number dependent on $V$, and will be determined later. And $\rho\in C_0^\infty(\mathbb{R})$ is an even function satisfying $$\rho(t)=1, |t|\leq 1/2, \hspace{2mm}{\rm and} \hspace{2mm}\rho(t)=0, |t|\geq 1.
$$
Because $\cos tP$ has finite propagation speed, the kernel of the operator $T_{\lambda,\mu}$ satisfies \[T_{\lambda,\mu}(x,y)=0, \ \ {\rm if}\ \ d_g(x,y)>\delta.\] So $T_{\lambda,\mu}$ is a ``local operator'' which will allow us to deal with critically singular potentials $V$ later. Next, if we use an identity between the resolvent operators of Laplacian and fractional Laplacian (see e.g. \cite{MSbook}, page 118, (5.28)), then for fixed $z=(\lambda+i)^{\al}$ we have
\begin{equation}
    \begin{aligned}
        ((-\lap)^{\alpha/2}-z)^{-1}&=\frac{z^{\frac{2-\alpha}{\alpha}}}{\alpha/2}(-\lap-z^{2/\alpha})^{-1}+\frac{\sin(\pi \alpha/2)}{\pi}\int_0^\infty\frac{\gamma^{\alpha/2}(\gamma-\lap)^{-1}}{\gamma^\alpha-2z\gamma^{\alpha/2}\cos(\pi \alpha/2)+z^2}d\gamma \\
        &:=\tfrac{(\lambda+i)^{(2-\alpha)}}{\alpha/2}T_{\lambda,1}+\tfrac{(\lambda+i)^{(2-\alpha)}}{\alpha/2}R_{\lambda,1} + \tilde{T}_{\lambda}+\tilde{R}_{\lambda}
    \end{aligned}
\end{equation}
where
\begin{equation}\label{Ttilde}
\tilde{T}_\lambda=\frac{\sin(\pi \alpha/2)}{-\pi\sqrt{\gamma}}\int_0^\infty\int_0^\infty \frac{\gamma^{\alpha/2}\rho(t/\delta)e^{-\sqrt{\gamma}t}\cos tP}{\gamma^\alpha-2z\gamma^{\alpha/2}\cos(\pi \alpha/2)+z^2}d\gamma dt
\end{equation}
\begin{equation}\label{Rtilde}
\tilde{R}_\lambda=\frac{\sin(\pi \alpha/2)}{-\pi\sqrt{\gamma}}\int_0^\infty\int_0^\infty \frac{\gamma^{\alpha/2}(1-\rho(t/\delta))e^{-\sqrt{\gamma}t}\cos tP}{\gamma^\alpha-2z\gamma^{\alpha/2}\cos(\pi \alpha/2)+z^2}d\gamma dt.
\end{equation}
So we can write the identity as
\begin{equation}\label{b2}
    \begin{aligned}
        I&=((-\lap)^{\alpha/2}-z)^{-1}((-\lap)^{\alpha/2}-z) \\
        &=\Big(\tfrac{(\lambda+i)^{(2-\alpha)}}{\alpha/2}T_{\lambda,1}+ \tilde{T}_\lambda\Big)((-\lap)^{\alpha/2}-z)+\Big(\tfrac{(\lambda+i)^{(2-\alpha)}}{\alpha/2}R_{\lambda,1} +\tilde{R}_\lambda\Big)((-\lap)^{\alpha/2}-z).
    \end{aligned}
\end{equation}
Consequently, we may reproduce any function $u\in C^\infty(M)$ in the following way:
\begin{equation}\label{b3}
\begin{aligned}
   u(x)=&\Big(\tfrac{(\lambda+i)^{(2-\alpha)}}{\alpha/2}T_{\lambda,1}+ \tilde{T}_\lambda\Big)\Big((-\lap)^{\alpha/2}-z+V\Big)u-\Big(\tfrac{(\lambda+i)^{(2-\alpha)}}{\alpha/2}T_{\lambda,1}+ \tilde{T}_\lambda\Big)(Vu) \\
   &+\Big(\tfrac{(\lambda+i)^{(2-\alpha)}}{\alpha/2}R_{\lambda,1} +\tilde{R}_\lambda\Big)((-\lap)^{\alpha/2}-z)u.
   \end{aligned}
\end{equation}
Throughout the paper, we fix $z=(\lambda+i)^\alpha$. In the next section, we shall study these operators in detail.
\section{$L^q-L^p $ norms of the operators}
In this section, we will prove several estimates about the $L^q-L^p$ norms of the operators $T_{\lambda,1}$, $\tilde T_\lambda$, $R_{\lambda,1}$, $\tilde R_\lambda$ defined in the previous section.

\begin{proposition}\label{p1}
   Let $n\ge2$ and $0<\alpha<2$. If \hspace{1mm}$T_{\lambda,1}$  is defined as in \eqref{regular resolvent}, $p_c=\frac{2(n+1)}{n-1}$, then for $\lambda\ge1$
  \begin{equation}\label{c1}
     \Big\| \tfrac{(\lambda+i)^{(2-\alpha)}}{\alpha/2}T_{\lambda,1} f\Big\|_{L^{p}(M)}\leq C_{\delta} \lambda^{1-\al+\sigma(p)}\|f\|_{L^{2}(M)}
  \end{equation}
provided that $p\in[p_c,\infty]$ if $n=2$ or $3$, $p\in[p_c,\infty)$ if $n=4$ and $p\in[p_c,\frac{2n}{n-4}]$ for $n\ge 5$,
  where $C_\delta$ is a constant dependent on $\delta$.
  Moreover, if $n\ge3$, $\frac{2n}{n+1}<\al<2$, $\frac{1}{p^*}-\frac{1}{p}=\frac{\alpha}{n}$, then for $p_c\le p < \frac{2n}{n+1-2\alpha}$, $\lambda\ge1$
   \begin{equation}\label{c2}
     \Big\|\tfrac{(\lambda+i)^{(2-\alpha)}}{\alpha/2}T_{\lambda,1} f\Big\|_{L^{p}(M)}\leq C_1 \|f\|_{L^{p^*}(M)}
  \end{equation}
  where $C_1$ is a constant independent of $\delta$. If $n=2$, then this inequality holds for $6\le p<\frac{4}{3-2\alpha}$ when $\frac43<\alpha<\frac32$, and for $6\le p<\infty$ when $\frac32\le \alpha<\frac53$.
\end{proposition}

\begin{proof}{}
To prove \eqref{c1}, we write
\begin{equation}\label{decomp} \tfrac{(\lambda+i)^{(2-\alpha)}}{\alpha/2}T_{\lambda,1} f=\tfrac{(\lambda+i)^{(2-\alpha)}}{\alpha/2}((-\lap-(\lambda+i)^{2})^{-1}-R_{\lambda,1})f
\end{equation} By using \cite{sogge88} or \cite[Proposition 2.4]{BSS19}, we obtain
 \[\|(-\lap-(\lambda+i)^{2})^{-1}f\|_{L^{p}(M)}\le C_\delta \lambda^{\sigma(p)-1}\|f\|_{L^{2}(M)}\]
provided that $p\in[p_c,\infty]$ if $n=2$ or 3, $p\in[p_c,\infty)$ if $n=4$ and $p\in[p_c,\frac{2n}{n-4}]$ for $n\ge 5$.

Since the multiplier associated to the  operator $R_{\lambda,1}(P)$ is
$$R_{\lambda,1}(\tau)=\frac{1}{i(\lambda+i)}\int_0^\infty(1-\rho(t/\delta))e^{i\lambda t-t}\cos t\tau dt,
$$
integration by part in the $t$ variable gives us the bound
\begin{equation}\label{mul}
R_{\lambda,1}(\tau)\leq C_{N,\delta}\lambda^{-1}\Big( (1+|\lambda-\tau|)^{-N}+(1+|\lambda+\tau|)^{-N} \Big).
\end{equation}
Let $\chi_\lambda=\sum\limits_{\lambda_j\in[\lambda-1,\lambda)} E_jf$ be the spectral projection operator associated with $P=\sqrt{-\Delta_g}$ corresponding to the unit interval $[\lambda-1,\lambda]$. By the classical spectral projection bounds in \cite{sogge88}, we get
\begin{align*}
    \|R_{\lambda,1} f\|_{L^{p}(M)}&\leq \sum\limits_{k=1}^\infty  \|\chi_kR_{\lambda,1} f\|_{L^{p}(M)} \\
    &\leq \sum\limits_{k=1}^\infty k^{\sigma(p)}\big(\sup_{\tau\in[k-1,k)}R_{\lambda,1}(\tau)\big)\|\chi_k f\|_{L^{2}(M)}\\
    &\leq C_{N,\delta}\sum\limits_{k=1}^\infty k^{\sigma(p)}\lambda^{-1}\Big( (1+|k-\lambda|)^{-N}+(1+|k+\lambda|)^{-N} \Big)\|f\|_{L^{2}(M)}\\
    &\leq C_{\delta}\lambda^{\sigma(p)-1}\|f\|_{L^{2}(M)}.
\end{align*}
Then \eqref{c1} follows by Minkowski inequality.

To prove \eqref{c2}, we will decompose the operator $T_{\lambda,1}$ in the following way. Fix a function $\beta\in C_0^\infty(\mathbb{R})$ satisfying
\begin{equation}
    \beta(t)=0,\hspace{1mm} t\notin [1/2,2],\hspace{1mm}  |\beta(t)|\leq 1,\hspace{1mm}  \sum\limits_{j=-\infty}\limits^{\infty}\beta(2^{-j}t)=1,\hspace{1mm}  t>0.
\end{equation}
Similar to \cite{BSSY}, we define operators
\begin{equation}
    S_jf=\frac{1}{i(\lambda+i)}\int_0^\infty\beta(\lambda2^{-j}t)\rho(t/\delta)e^{i\lambda t-t}\cos tP fdt,
\end{equation}
and
\begin{equation}
    S_0f=\frac{1}{i(\lambda+i)}\int_0^\infty\tilde{\rho}(\lambda t)\rho(t/\delta)e^{i\lambda t-t}\cos tP fdt,
\end{equation}
with
$$ \tilde{\rho}(t)=\big(1-\sum\limits_{j=0}\limits^{\infty}\beta(2^{-j}t)\big)\in C_0^\infty(\mathbb{R}).
$$
Clearly $\tilde{\rho}(t)=0$ if $t>2$ and  $\tilde{\rho}(t)=1$ if $0<t<1$.

By definition and the support property, we have
\begin{equation}\label{daa}
T_{\lambda,1} f=\sum\limits_{j=0}^{j_0} S_jf
\end{equation}
where $\lambda^{-1}2^{j_0}\approx \delta$.
Since the multiplier associated to the operator $S_0$ is
\begin{equation}
    S_0(\tau)=\frac{1}{i(\lambda+i)}\int_0^\infty\tilde{\rho}(\lambda t)\rho(t/\delta)e^{i\lambda t-t}\cos t\tau dt
\end{equation}
 it is not hard to prove that it satisfies for $j=0,1,2,...$
\begin{equation}\label{8.4}
|\tfrac{d^j}{d\tau^j}S_0(\tau)|\le \begin{cases}
  C_j\lambda^{-2-j}\ \ \ \ \ \ {\rm if}\ \ |\tau|\le\lambda,\\
  C_j|\tau|^{-2-j}\ \ \ \ \ {\rm if}\ \ |\tau|>\lambda.
\end{cases}
\end{equation}
by the $O(\lambda^{-1})$ smallness of the time support and integration by parts (see e.g. \cite[Lemma 2.2]{sy}).
So the multiplier associated to the operator $\frac{(\lambda+i)^{(2-\alpha)}}{\alpha/2}S_0$  satisfies $$ |\tfrac{(\lambda+i)^{(2-\alpha)}}{\alpha/2}\tfrac{d^j}{d\tau^j}S_0(\tau)|\le C_j (1+|\tau|)^{-\alpha-j},\ \ j=0,1,2,..$$ Thus $\frac{(\lambda+i)^{(2-\alpha)}}{\alpha/2}S_0$ is a pseudo-differential operator of order $-\alpha$ (see e.g. \cite[Theorem 4.3.1]{fio}). Then this leads  to the following kernel estimate (see e.g. Proposition 1 on the page 241 of \cite{stein})
\begin{equation}\label{S0kernel}
  |\tfrac{(\lambda+i)^{(2-\alpha)}}{\alpha/2}S_0(x,y)| \le C d_g(x,y)^{\alpha-n}.
\end{equation}

By the finite propagation speed property, $S_0(x,y)$ is supported in $\{d_g(x,y)\le 2\lambda^{-1}\}$. By our assumptions, $1-(\frac1{p^*}-\frac1{p})=\frac{n-\alpha}{n}$, $1<p^*<p<\infty$, then by Hardy-Littlewood-Sobolev inequality
\begin{equation}
     \|  \tfrac{(\lambda+i)^{(2-\alpha)}}{\alpha/2}S_0f\|_{L^{p}(M)}\leq C \|f\|_{L^{p^*}(M)}.
  \end{equation}

To deal with the operators $S_j$, we proceed as in \cite{sy}. Note that their argument is only used in dimension $n\ge3$ to prove uniform estimates for a local operator, but their kernel estimates can be directly extended to dimension $n=2$, since these estimates are obtained from the Hadamard parametrix and integration by parts.

First, as in \cite[(2.18)]{sy}, by using the Hadamard parametrix for the operator $\cos tP$ and a scaling argument, the kernel $S_j(x,y)$ is the sum of the following two terms
$$ S_j^{\pm}(x,y)=\frac{\varepsilon^{1-n}}{i(\lambda+i)}\int_0^\infty\int_{\mathbb{R}^n}\beta(t)\rho(\varepsilon\delta^{-1}t)e^{-\varepsilon t}\alpha_{\pm}(\varepsilon t,x,y,|\xi|/\varepsilon)e^{i\frac{\kappa(x,y)}{\varepsilon}\cdot\xi\pm t|\xi|}d\xi dt
$$
where $\alpha_{\pm}(t,x,y,|\xi|)$ are 0-order symbol functions in $\xi$, $\varepsilon=\frac{2^j}{\lambda}$ is a dyadic number between $\lambda^{-1}$ and $\delta$, and
$ S_j^{\pm}(x,y)=0$ if $d_g(x,y)\ge 4\varepsilon$.

If $d_g(x,y)/ \varepsilon\le 1/4$, an integration by parts argument with respect to $\xi$ would show that
$$ |S_j^\pm(x,y)|\le C\eps ^{1-n}\lambda^{-1}.
$$
Since $1-\alpha<0$, by using Young's inequality and summing over a geometric series, we have in this case
$$  \|  \tfrac{(\lambda+i)^{(2-\alpha)}}{\alpha/2}\sum\limits_{j=0}^{j_0}S_j^\pm f\|_{L^{p}(M)}\le C\sum_{j=0}^{j_0}2^{(1-\alpha)j}\|f\|_{L^{p^*}(M)}\le C \|f\|_{L^{p^*}(M)}.
$$
Next, we are reduced to considering the operator $K_j^\pm$ with kernel $K_j^\pm(x,y):=\beta(\frac{d_g(x,y)}{2\varepsilon})S_j^\pm(x,y)$ in which $\beta(r)$ is supported when $r\in(1/2,2)$. By calculating the kernel explicitly as in \cite[(2.21), (2.23)]{sy}, we can write the kernel $K_j^\pm(x,y)$ as
\begin{equation}\label{Kjkernel}
  K_j^\pm(x,y)=2^{-j}\varepsilon^{2-n}a_{1,w}(x,y)+T_j(x,y),
\end{equation}
with  \begin{equation}\label{Tjkernel}T_j(x,y)=\lambda^{\frac{n-3}2}\eps^{-\frac{n-1}2}e^{i\lambda d_g(x,y)}a_\eps(x,y),\end{equation}
where $a_{1,w}(x,y)$ is a uniformly bounded smooth function with support when $d_g(x,y)/\varepsilon\in(1/4,4)$, and the smooth function $a_\eps(x,y)$ is supported when $d_g(x,y)/\eps\in(1/4,4)$, and
 \[|\partial_{x,y}^\gamma a_\eps(x,y)|\le C_\gamma \eps^{-|\gamma|}.\]Since $2^{-j}\varepsilon^{2-n}= \varepsilon ^{1-n}\lambda^{-1}$, we can deal with the first term of \eqref{Kjkernel} exactly the same as above by Young's inequality. As in \cite[(2.27)]{sy}, by using the standard Carleson-Sj\"olin estimate in \cite[Theorem 2.2.1]{fio}, the operator $T_j$ associated with the kernel \eqref{Tjkernel} satisfies the following bound
\begin{equation}\label{c3}
    \|T_j f\|_{L^{p}(M)}\leq C \lambda^{-2+n(\frac1q-\frac1p)}2^{j(\frac{n+1}{2}-\frac n{q})}\|f\|_{L^{q}(M)}
\end{equation}
given that $(\frac1q,\frac1p)$ is on the line segment $\{\frac1p=\frac{n-1}{n+1}(1-\frac1q),\ \frac12\leq \frac1q\leq1\}$. See Figure \ref{fig1}.
\begin{figure}
   \includegraphics[width=0.8\textwidth]{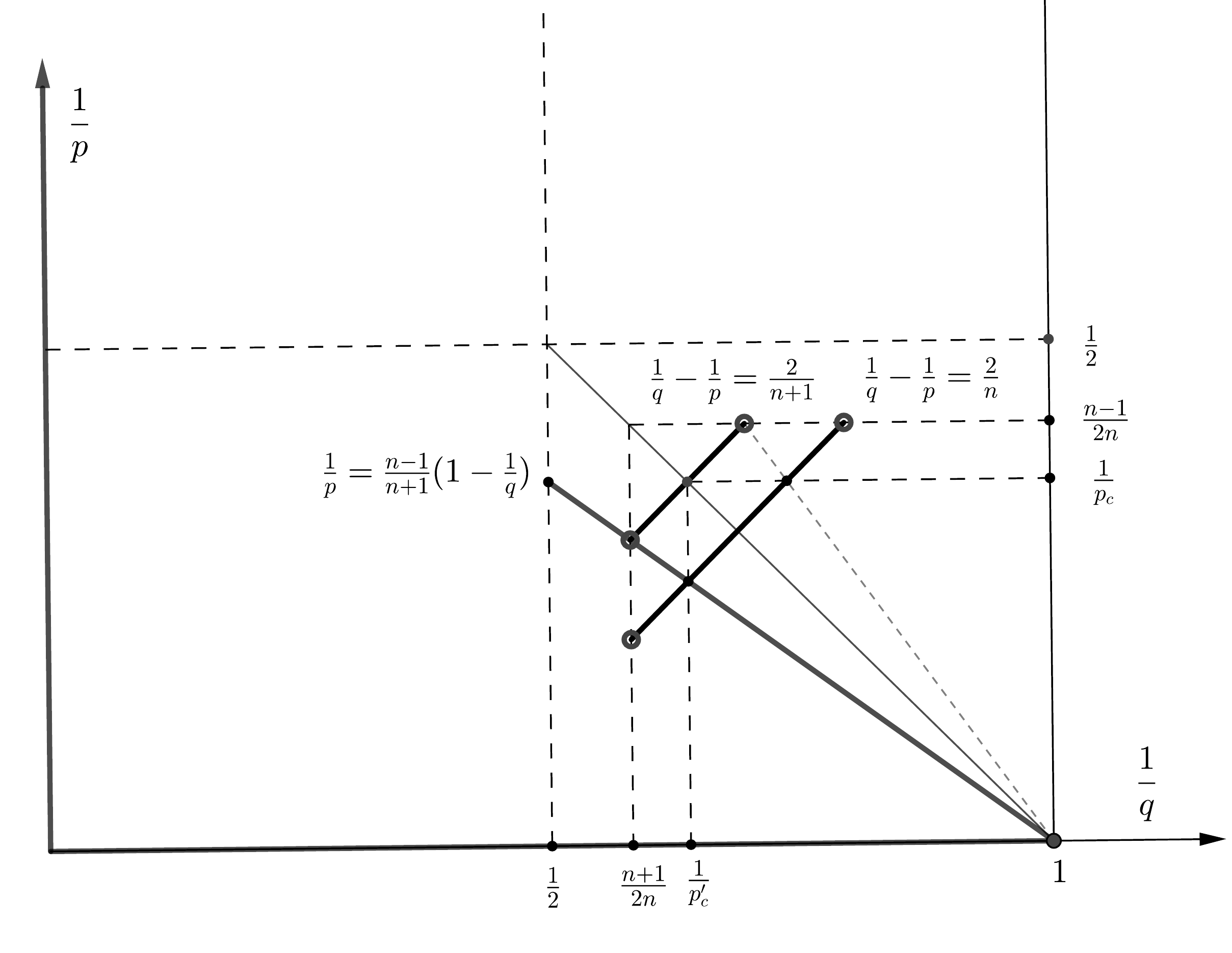}
   \caption{Interpolation: $n\ge3$}
   \label{fig1}
   \end{figure}
\begin{figure}
   \includegraphics[width=0.8\textwidth]{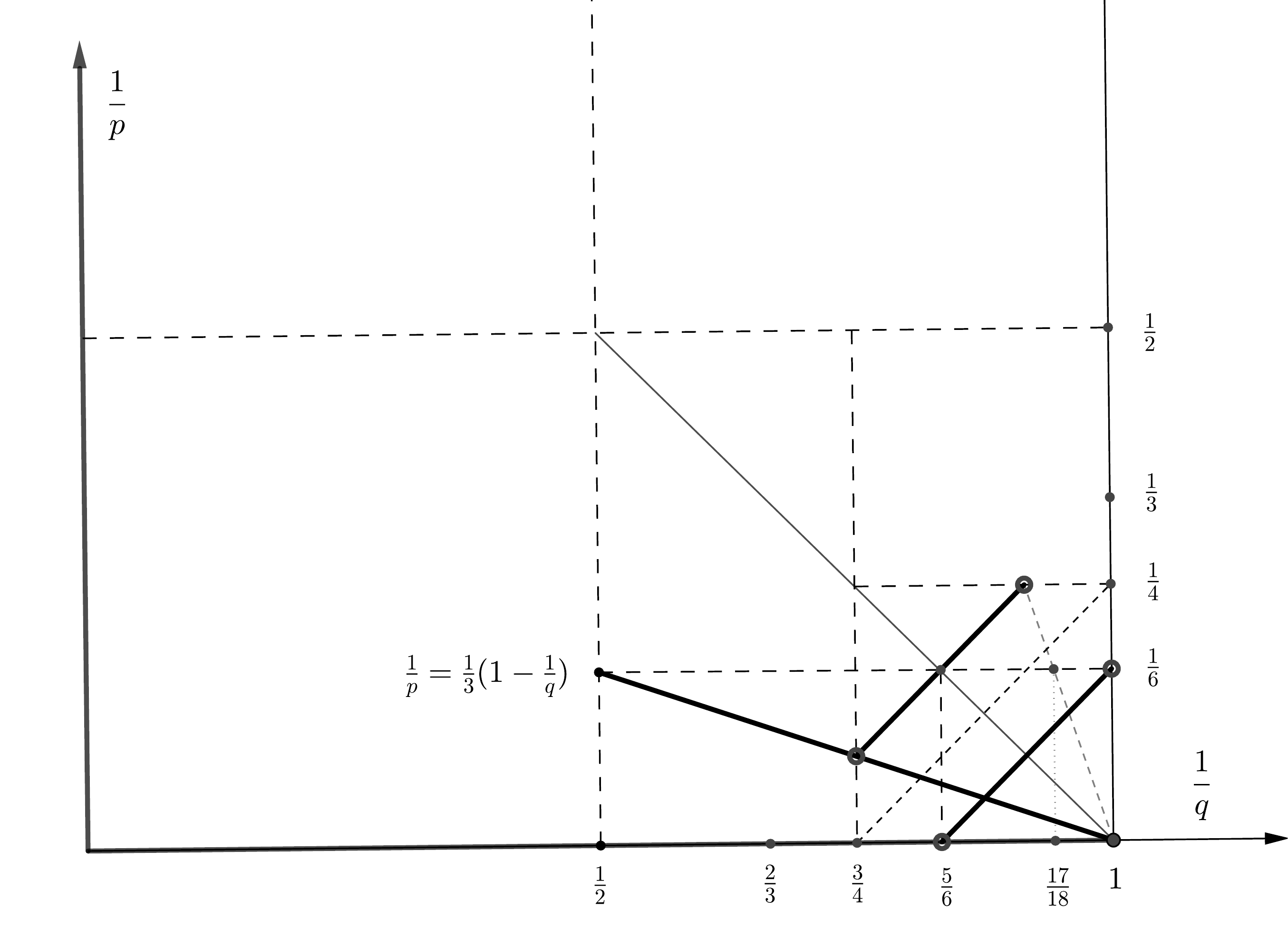}
   \caption{Interpolation: $n=2$}
   \label{fig2}
   \end{figure}

On the other hand, by Young's inequality, $T_j$ satisfies the same $L^q-L^p$ bound as in \eqref{c3} on the line $(\frac1q,0)$. Note that $\frac1q>\frac{n+1}{2n}$ ensures the convergence of the geometric series from summing \eqref{c3}.
By a simple calculation, when $\frac{2n}{n+1}<\alpha<2$, the fractional Sobolev line $ \frac1{q}-\frac1{p}=\frac \alpha n$ can intersect the line segment $\{\frac1p=\frac{n-1}{n+1}(1-\frac1q),\ \frac{n+1}{2n}<\frac1q\leq1\}$. Moreover, the fractional Sobolev line intersects the line $(\frac{n+1}{2n},\frac1p)$ at the point $(\frac{n+1}{2n},\frac{n+1-2\alpha}{2n})$. So by summing over a geometric series, and by a simple interpolation and duality argument, we get when $n\ge3$ and $\frac{2n}{n+1}<\alpha<2$
$$  \|  \tfrac{(\lambda+i)^{(2-\alpha)}}{\alpha/2}\sum\limits_{j=0}^{j_0}T_j f\|_{L^{p}(M)}\le C  \|f\|_{L^{p^*}(M)}
$$
if \[\tfrac{n+1-2\alpha}{2n}<\tfrac1{p}<\tfrac{n-1}{2n},\ \tfrac1{p^*}-\tfrac1p=\tfrac\alpha n.\]
Since $\frac{n+1-2\alpha}{2n}<\frac1{p_c}<\frac{n-1}{2n}$, we complete the proof for $n\ge3$. Similarly, when $n=2$, $p_c=6$, we still have this inequality for $6\le p<\frac{4}{3-2\alpha}$ if $\frac43<\alpha<\frac32$, and for $6\le p<\infty$ if $\frac32\le \alpha<\frac53$. Here the upper bound $\frac53$ is needed to ensure $p^*>1$ for $p\ge 6$. See Figure \ref{fig2}. So the proof is finished.
\end{proof}
\begin{proposition}\label{p2}
  Let $n\ge2$ and $0<\alpha<2$. If \hspace{1mm}$\tilde{T}_\lambda$  is defined as in \eqref{Ttilde}, $p_c=\frac{2(n+1)}{n-1}$, then for $\lambda \geq 1$,
  \begin{equation}\label{c4}
     \| \tilde{T}_\lambda f\|_{L^{p}(M)}\leq C_{\delta} \lambda^{1/2-\alpha+\sigma(p)}\|f\|_{L^{2}(M)}
  \end{equation}
  provided that $p\in[p_c,\infty]$ if $n=2$, $1<\alpha<2$ $($or $n=3$, $\frac32<\alpha<2)$, and $p\in[p_c,\frac{2n}{n-2\alpha})$ if  $n\ge2\alpha$, $\frac{n}{n+1}<\alpha<2$, where $C_\delta$ is a constant dependent on $\delta$.
 Moreover, if $n\ge3$, $0<\alpha<2$ $($or $n=2$, $0<\alpha<\frac53$$)$, $p_c\le p<\infty$, $\frac{1}{p^*}-\frac{1}{p}=\frac{\alpha}{n}$, then for $\lambda\ge \delta^{-(n+1)/\alpha}$
   \begin{equation}\label{c5}
     \| \tilde{T}_\lambda f\|_{L^{p}(M)}\leq C_{2} \|f\|_{L^{p^*}(M)}
  \end{equation}
  where $C_{2}$ is a constant independent of $\delta$.
\end{proposition}

\begin{proof}{}
To prove \eqref{c4}, we write
\begin{equation}\label{c6}
\begin{aligned}
    \tilde{T}_\lambda f&=\frac{\sin(\pi \alpha/2)}{\pi}\int_0^\infty\frac{\gamma^{\alpha/2}(\gamma-\lap)^{-1}f}{\gamma^\alpha-2z\gamma^{\alpha/2}\cos(\pi \alpha/2)+z^2}d\gamma -\tilde{R}_\lambda f \\
    &:=Kf -\tilde{R}_\lambda f
    \end{aligned}
\end{equation}
where $\tilde{R}_\lambda$ is defined in \eqref{Rtilde}. Here we need some observations in \cite{HSZ}. Note that
\[\gamma^\alpha-2z\gamma^{\alpha/2}\cos(\pi \alpha/2)+z^2=(\gamma^{\alpha/2}-ze^{i\pi \alpha/2})(\gamma^{\alpha/2}-ze^{-i\pi \alpha/2}).\]
Since $z=(\lambda+i)^\alpha$ and $0<\arg(z)<\pi \alpha/4$ when $\lambda$ is large, we have $|\arg(ze^{\pm i\pi \alpha/2})|>\pi \alpha/4$. Then
\[|\gamma^{\alpha/2}-ze^{\pm i\pi \alpha/2}|\approx \gamma^{\alpha/2}+|z|\]
which implies
\begin{equation}
|\gamma^\alpha-2z\gamma^{\alpha/2}\cos(\pi \alpha/2)+z^2|\approx \gamma^{\alpha}+|z|^2.
\end{equation}
Moreover, by using the proof of \cite[Proposition 2.4]{BSS19} or \cite[Lemma 3.2]{FKS}, it is not hard  to see that
$$ \|(\gamma-\lap)^{-1}f\|_{L^{p}(M)}\le C_\delta \gamma^{-\frac{3}4+\frac{\sigma(p)}{2}}\|f\|_{L^{2}(M)}
$$
provided that $p\in[p_c,\infty]$ if $n=2$ or $3$, $p\in[p_c,\infty)$ if $n=4$ and $p\in[p_c,\frac{2n}{n-4}]$ for $n\ge 5$.
So by Minkowski inequality, if $\alpha>\frac12+\sigma(p)$, then
\begin{equation}\label{K2p}
    \begin{aligned}
   \| Kf\|_{L^{p}(M)}&\leq C_\delta\int_0^\infty\frac{\gamma^{\alpha/2}\gamma^{-\frac{3}4+\frac{\sigma(p)}{2}}}{\gamma^{\alpha}+|z|^2}d\gamma\|f\|_{L^{2}(M)} \\
   &\leq C_\delta \lambda^{1/2-\alpha+\sigma(p)}\|f\|_{L^{2}(M)}.
    \end{aligned}
\end{equation}
 Note that $\alpha>\frac12+\sigma(p)$ holds provided that $p\in[p_c,\infty]$ if $n=2$ and $1<\alpha<2$ $($or $n=3$ and $\frac32<\alpha<2)$, and $p\in[p_c,\frac{2n}{n-2\alpha})$ if $n\ge2\alpha$, $\frac{n}{n+1}<\alpha<2$.

 Next, for the operator $\tilde{R}_\lambda$, we first consider the operator
$$R_{0,\sqrt{\gamma}}=\frac{1}{-\sqrt{\gamma}}\int_0^\infty (1-\rho(t/\delta))e^{-\sqrt{\gamma}t}\cos tP dt.
$$
The multiplier associated to the operator $R_{0,\sqrt{\gamma}}$  is
$$R_{0,\sqrt{\gamma}}(\tau)=\frac{1}{-\sqrt{\gamma}}\int_0^\infty (1-\rho(t/\delta))e^{-\sqrt{\gamma}t}\cos t\tau dt.
$$
Integration by parts in the $t$ variable gives us
\begin{equation}\label{mul2}
|R_{0,\sqrt{\gamma}}(\tau)|\leq C_{N}\delta^{-N}\gamma^{-1} (1+|\sqrt{\gamma}+\tau|)^{-N},\ N=1,2,3,...
\end{equation}
If we use the spectral projection bounds for $\chi_k$ in \cite{sogge88}, then
\begin{equation}
    \begin{aligned}
    \|R_{0,\sqrt{\gamma}}f\|_{L^{p}(M)}&\leq \sum\limits_{k=1}^\infty  \|\chi_kR_{0,\sqrt{\gamma}} f\|_{L^{p}(M)} \\
    &\leq \sum\limits_{k=1}^\infty k^{\sigma(p)}\big(\sup_{\tau\in[k-1,k)}R_{0,\sqrt{\gamma}}(\tau)\big)\|\chi_k f\|_{L^{2}(M)}\\
    &\leq C_{N}\delta^{-N}\sum\limits_{k=1}^\infty k^{\sigma(p)}\gamma^{-1} (1+|k+\sqrt{\gamma}|)^{-N}\|f\|_{L^{2}(M)}\\
    &\leq C_{\delta}\gamma^{-1}\|f\|_{L^{2}(M)}
    \end{aligned}
    \nonumber
\end{equation}
So we have
\begin{equation}
    \begin{aligned}
    \|\tilde{R}_\lambda f\|_{L^{p}(M)}&\leq C_\delta\int_0^\infty\frac{\gamma^{\alpha/2}\gamma^{-1}}{\gamma^{\alpha}+|z|^2}d\gamma\|f\|_{L^{2}(M)} \\
   &\leq C_\delta \lambda^{-\alpha}\|f\|_{L^{2}(M)} \qquad\text{better than \eqref{K2p}}.
    \end{aligned}
\end{equation}
Consequently, \eqref{c4} follows from Minkowski inequality.

To prove \eqref{c5}, we use the same strategy, as well as the following lemma.
\begin{lemma}\label{lemma1}
Let $n\ge2$ and $0<\alpha<2$. If $K$ is defined as in \eqref{c6}, then for $1< p^*< p<\infty$, $\frac{1}{p^*}-\frac1p=\frac \alpha n$,
 \begin{equation}\label{Kunif}
     \| K f\|_{L^{p}(M)}\leq C\|f\|_{L^{p^*}(M)}
  \end{equation}
  where $C$ is a constant independent of $\delta$.
\end{lemma}
Unlike the $L^2-L^p$ bound \eqref{K2p}, this lemma cannot directly follow  from the resolvent estimates and Minkowski inequality. We will prove it by using Li-Yau's heat kernel estimates and Sogge's spectral projection bounds. We postpone the proof of this lemma, and finish proving the proposition first. If $n\ge3$, $0<\alpha<2$, and $p_c\le p<\infty$, then there is a unique $p^*\in(1,p)$ such that $\frac1{p^*}-\frac1p=\frac{\alpha}{n}$. Note that it is also true for $n=2$, $0<\alpha<\frac53$ and $p_c\le p<\infty$. For the operator $\tilde{R}_\lambda$, we use \eqref{mul2} and the spectral projection bounds twice:
\begin{equation}
    \begin{aligned}
    \|R_{0,\sqrt{\gamma}} f\|_{L^{p}(M)}&\leq \sum\limits_{k=1}^\infty  \|\chi_kR_{0,\sqrt{\gamma}} f\|_{L^{p}(M)} \\
    &\leq \sum\limits_{k=1}^\infty k^{\sigma(p)}\big(\sup_{\tau\in[k-1,k)}R_{0,\sqrt{\gamma}}(\tau)\big)\|\chi_k f\|_{L^{2}(M)}\\
    &\leq C_{N}\delta^{-N}\sum\limits_{k=1}^\infty k^{\sigma(p)+\sigma(p^*)}\gamma^{-1} (1+|k+\sqrt{\gamma}|)^{-N}\|f\|_{L^{p^*}(M)}\\
    &\leq C\delta^{-n-1}\gamma^{-1}\|f\|_{L^{p^*}(M)} \qquad\text{for}\  N=n+1
    \end{aligned}
    \nonumber
\end{equation}
since $\sigma(p)+\sigma(p^*)\le n-1$.
Then we have
\begin{equation}
    \begin{aligned}
    \|\tilde{R}_\lambda f\|_{L^{p}(M)}&\leq C\frac{\sin(\pi \alpha/2)}{\pi}\int_0^\infty\frac{\gamma^{\alpha/2}\gamma^{-1}\delta^{-n-1}}{\gamma^{\alpha}+|z|^2}d\gamma\|f\|_{L^{2}(M)} \\
   &\leq C\delta^{-n-1}\lambda^{-\alpha}\|f\|_{L^{p^*}(M)} \\
   & \leq C \|f\|_{L^{p^*}(M)} \qquad \text{for $\lambda\ge\delta^{-(n+1)/\alpha}$}.
    \end{aligned}
\end{equation}
So the proof of Proposition \ref{p2} is complete.
\end{proof}
\begin{proof}[Proof of Lemma \ref{lemma1}]
In order to use heat kernel, we write the resolvent operator as
\begin{equation}
\begin{aligned}
(\gamma-\lap)^{-1}&=\int_0^\infty e^{-\gamma t-(-\lap)t}dt\\
&=\int_0^1 e^{-\gamma t-(-\lap)t}dt+\int_1^\infty e^{-\gamma t-(-\lap)t}dt \\
&:=H_0+H_1.
\end{aligned}
\end{equation}
The operator $H_1$ is a good error term. First, by functional calculus, $$H_1=e^{-\gamma -(-\lap)} (\gamma-\lap)^{-1}.$$
 The  multiplier associated to the operator $H_1$ is equal to $H_1(\tau)=e^{-\gamma -\tau^2} (\gamma+\tau^2)^{-1}.$
Then we use the spectral projection bounds again
 \begin{equation}
    \begin{aligned}
    \|H_1f\|_{L^{p}(M)}&\leq \sum\limits_{k=1}^\infty  \|\chi_kH_1 f\|_{L^{p}(M)} \\
    &\leq \sum\limits_{k=1}^\infty k^{\sigma(p)}\big(\sup_{\tau\in[k-1,k)}H_1(\tau)\big)\|\chi_k f\|_{L^{2}(M)}\\
    &\leq \sum\limits_{k=1}^\infty k^{\sigma(p)+\sigma(p^*)}e^{-\gamma}e^{-(k-1)^2}(\gamma+(k-1)^2)^{-1}\|f\|_{L^{p^*}(M)}\\
    &\leq Ce^{-\gamma}\|f\|_{L^{p^*}(M)}.
    \end{aligned}
    \nonumber
\end{equation}
So we have
\begin{equation}
\begin{aligned}
\|&\frac{\sin(\pi \alpha/2)}{\pi}\int_0^\infty\frac{\gamma^{\alpha/2}H_1f}{\gamma^\alpha-2z\gamma^{\alpha/2}\cos(\pi \alpha/2)+z^2}d\gamma\|_{L^{p}(M)} \\
&\leq C\int_0^\infty\frac{\gamma^{\alpha/2}e^{-\gamma}}{\gamma^{\alpha}+|z|^2}d\gamma\|f\|_{L^{p^*}(M)} \\
&\leq C\lambda^{-2\alpha}\|f\|_{L^{p^*}(M)}  \quad {\rm better\ than}\ \eqref{Kunif}.
\end{aligned}
\end{equation}
Now we only need to consider $H_0$. Recall Li-Yau's upper bounds on the heat kernel \cite{LY}
$$ |e^{t\lap}(x,y)|\leq C_0 t^{-\frac n2}  e^{-c_0d_g^2(x,y)/t},\ 0<t\le 1.$$
By Minkowski inequality, we can bound the operator kernel by the double integral
\begin{equation}
\begin{aligned}
\Big|\tfrac{\sin(\pi \alpha/2)}{\pi}\int_0^\infty\frac{\gamma^{\alpha/2}H_0(x,y)}{\gamma^\alpha-2z\gamma^{\alpha/2}\cos(\pi \alpha/2)+z^2}d\gamma\Big|  \\
\leq C\int_0^1\int_0^\infty\frac{\gamma^{\alpha/2}e^{-\gamma t}t^{-\frac n2}  e^{-c_0d_g^2(x,y)/t}}{\gamma^{\alpha}+|z|^2}d\gamma dt.
\end{aligned}
\end{equation}
We will divide the kernel  into four parts.
\qquad\newline

\noindent i) $|z|^{-\frac2\alpha}\le t\le 1, \ d_g(x,y)\le |z|^{-\frac1\alpha}$

In this case
\begin{equation}
\begin{aligned}
K_1(x,y)&\leq C\int_{|z|^{-\frac2\alpha}}^1\int_0^\infty\frac{\gamma^{\alpha/2}e^{-\gamma t}t^{-\frac n2}  e^{-c_0d_g^2(x,y)/t}}{\gamma^{\alpha}+|z|^2}d\gamma dt  \\
&\leq C |z|^{-2}\int_{|z|^{-\frac2\alpha}}^1t^{-\frac \alpha2-1-\frac n2} dt \\
&\leq C |z|^{-1+\frac n\alpha}.
\end{aligned}
\end{equation}
If we let $\frac 1r=1-(\frac1{p^*}-\frac1{p})=\frac{n-\alpha}{n}$, then by Young's inequality we have
\begin{equation}
     \| \int K_1(x,y)f(y)dy\|_{L^{p}(M)}\leq C \|f\|_{L^{p^*}(M)}.
  \end{equation}
\noindent ii) $|z|^{-\frac2\alpha}\le t\le 1,\ d_g(x,y)\ge |z|^{-\frac1\alpha}$

  In this case
  \begin{equation}
\begin{aligned}
K_2(x,y) &\leq C\int_{|z|^{-\frac2\alpha}}^1\int_0^\infty\frac{\gamma^{\alpha/2}e^{-\gamma t}t^{-\frac n2}  e^{-c_0d_g^2(x,y)/t}}{\gamma^{\alpha}+|z|^2}d\gamma dt  \\
&\leq C |z|^{-2}\int_0^1\int_0^\infty\gamma^{\alpha/2}e^{-\gamma t}t^{-\frac n2}  e^{-c_0d_g^2(x,y)/t}d\gamma dt  \\
&\leq C|z|^{-2}\int_0^1t^{-\frac n2-\frac \alpha2-1}  e^{-c_0d_g^2(x,y)/t} dt \\
&\leq C |z|^{-2} {d_g(x,y)}^{-n-\alpha}.
\end{aligned}
\end{equation}
 By Young's inequality, we have
\begin{equation}
     \| \int K_2(x,y)f(y)dy\|_{L^{p}(M)}\leq C\|f\|_{L^{p^*}(M)}.
  \end{equation}

\noindent iii) $0<t<|z|^{-\frac2\alpha},\ d_g(x,y)\leq |z|^{-\frac1\alpha}$

 In this case
  \begin{equation}
\begin{aligned}
K_3(x,y) &\leq C\int_0^{|z|^{-\frac2\alpha}}\int_0^\infty\frac{\gamma^{\alpha/2}e^{-\gamma t}t^{-\frac n2}  e^{-c_0d_g^2(x,y)/t}}{\gamma^{\alpha}+|z|^2}d\gamma dt  \\
&\leq C \int_0^1\int_0^\infty\gamma^{-\alpha/2}e^{-\gamma t}t^{-\frac n2}  e^{-c_0d_g^2(x,y)/t}d\gamma dt  \\
&\leq C \int_0^1t^{-\frac n2+\frac \alpha2-1}  e^{-c_0d_g^2(x,y)/t} dt \\
&\leq C  {d_g(x,y)}^{-n+\alpha}.
\end{aligned}
\end{equation}
Since $\frac{1}{p^*}-\frac{1}{p}=\frac{\alpha}{n}$, by Hardy-Littlewood-Sobolev inequality, we have
\begin{equation}
     \| \int K_3(x,y)f(y)dy\|_{L^{p}(M)}\leq C \|f\|_{L^{p^*}(M)}.
  \end{equation}

\noindent iv) $0<t<|z|^{-\frac2\alpha},\ d_g(x,y)\geq |z|^{-\frac1\alpha}$

 In this case
  \begin{equation}
\begin{aligned}
K_4(x,y) &\leq C\int_0^{|z|^{-\frac2\alpha}}\int_0^\infty\frac{\gamma^{\alpha/2}e^{-\gamma t}t^{-\frac n2}  e^{-c_0d_g^2(x,y)/t}}{\gamma^{\alpha}+|z|^2}d\gamma dt  \\
&\leq C \int_0^{|z|^{-\frac2\alpha}}\int_0^\infty\gamma^{-\alpha/2}e^{-\gamma t}t^{-\frac n2}  e^{-c_0d_g^2(x,y)/t}d\gamma dt  \\
&\leq C\int_0^{|z|^{-\frac2\alpha}}t^{-\frac n2+\frac \alpha2-1}  e^{-c_0d_g^2(x,y)/t} dt\\
&\leq C |z|^{\frac{n-\alpha}{\alpha}}  \int_0^1t^{-\frac n2+\frac \alpha2-1}  e^{-c_0|z|^{\frac2\alpha} d_g^2(x,y)/t} dt \\
&\leq C |z|^{\frac{n-\alpha}{\alpha}}  e^{-c_0|z|^{2/\alpha} d_g^2(x,y)/2}.
\end{aligned}
\end{equation}
By Young's inequality, we have
\begin{equation}
     \| \int K_4(x,y)f(y)dy\|_{L^{p}(M)}\leq C \|f\|_{L^{p^*}(M)}.
  \end{equation}
So we finish the proof of Lemma \ref{lemma1}.
\end{proof}
\begin{proposition}\label{p3}
  Let $n\ge2$ and $0<\alpha<2$. If \hspace{1mm}$R_{\lambda,1}$  is defined as in \eqref{regular resolvent}, $p_c=\frac{2(n+1)}{n-1}$, then for $\lambda \geq 1$, $p_c\le p \le \infty$
  \begin{equation}\label{c7}
     \| \tfrac{(\lambda+i)^{(2-\alpha)}}{\alpha/2}R_{\lambda,1}((-\lap)^{\alpha/2}-z) f\|_{L^{p}(M)}\leq C_{\delta} \lambda^{\sigma(p)}\|f\|_{L^{2}(M)}.
  \end{equation}
  Similarly, for the operator $\tilde{R}_\lambda$ defined as in \eqref{Rtilde}
    \begin{equation}\label{c8}
     \| \tilde{R}_\lambda((-\lap)^{\alpha/2}-z) f\|_{L^{p}(M)}\leq C_{\delta}\|f\|_{L^{2}(M)}.
  \end{equation}
\end{proposition}
\begin{proof}
To prove \eqref{c7}, if we recall \eqref{mul} and use the spectral projection bounds \cite{sogge88}
\begin{align*}\|R_{\lambda,1} ((-\lap)^{\alpha/2}-z)f\|_{L^{p}(M)}&\leq \sum\limits_{k=1}^\infty  \|\chi_kR_{\lambda,1}((-\lap)^{\alpha/2}-z) f\|_{L^{p}(M)} \\
    &\leq \sum\limits_{k=1}^\infty k^{\sigma(p)}(k^\alpha-(\lambda+i)^\alpha)\big(\sup_{\tau\in[k-1,k)}R_{\lambda,1}(\tau)\big)\|\chi_k f\|_{L^{2}(M)}\\
    &\leq C_{N,\delta}\sum\limits_{k=1}^\infty k^{\sigma(p)}\lambda^{-1}(k^\alpha-(\lambda+i)^\alpha)\Big( (1+|k-\lambda|)^{-N}+(1+|k+\lambda|)^{-N} \Big)\|f\|_{L^{2}(M)}\\
    &\leq C_{\delta}\lambda^{\sigma(p)-2+\alpha}\|f\|_{L^{2}(M)}.\end{align*}
\noindent To prove \eqref{c8}, we recall \eqref{mul2} and use the spectral projection bounds again
    \begin{align*}
    \|R_{0,\sqrt{\gamma}} ((-\lap)^{\alpha/2}-z)f\|_{L^{p}(M)}&\leq \sum\limits_{k=1}^\infty  \|\chi_kR_{0,\sqrt{\gamma}} ((-\lap)^{\alpha/2}-z)f\|_{L^{p}(M)} \\
    &\leq \sum\limits_{k=1}^\infty k^{\sigma(p)}(k^\alpha-(\lambda+i)^\alpha)\big(\sup_{\tau\in[k-1,k)}R_{0,\sqrt{\gamma}}(\tau)\big)\|\chi_k f\|_{L^{2}(M)}\\
    &\leq C_{N,\delta}\sum\limits_{k=1}^\infty k^{\sigma(p)}(k^\alpha-(\lambda+i)^\alpha)\gamma^{-1} (1+|k+\sqrt{\gamma}|)^{-N}\|f\|_{L^{2}(M)}\\
    &\leq C_{\delta}\gamma^{-1}\lambda^\alpha\|f\|_{L^{2}(M)}.
    \end{align*}
So we have
\begin{equation}
    \begin{aligned}
    \|\tilde{R}_\lambda((-\lap)^{\alpha/2}-z)f\|_{L^{p_c}(M)}&\leq C_\delta\frac{\sin(\pi \alpha/2)}{\pi}\int_0^\infty\frac{\gamma^{\alpha/2}\gamma^{-1}\lambda^\alpha}{\gamma^{\alpha}+|z|^2}d\gamma\|f\|_{L^{2}(M)} \\
   &\leq C_\delta \|f\|_{L^{2}(M)}.
    \end{aligned}
\end{equation}
So the proof is finished.\end{proof}
\section{Proof of Theorem \ref{quasi1}}
We need to assume $n\ge4$, since in this case $V\in L^{n/\alpha}\subset L^2$ and the right hand side of \eqref{quasi} makes sense. By a simple interpolation, it suffices to prove it for $p\in[p_c,\frac{2n}{n+1-2\alpha})$. For each small $\delta>0$ choose a maximal $\delta$-separated collection of points $x_j\in M, j=1,...,N_\delta, N_\delta\approx \delta^{-n}$.
Thus, $M=\cup B_j$ if $B_j$ is the $\delta$-ball about $x_j$, and if $B_j^*$ is the $2\delta$-ball about the same center
 \begin{equation}\label{d1}
 \sum\limits_{j=1}\limits^{N_\delta}\textbf{1}_{B_j^*}(x)\leq C_{M}
 \end{equation}
where $C_M$ is independent of $\delta$  if $\textbf{1}_{B_j^*}$ denotes the indicator function of $B_j^*$.  Since $V\in L^{n/\alpha}(M)$
we can fix $\delta>0$ small enough so that
\begin{equation}\label{d2}
 C_{M} \Big(C_0\sup\limits_{x\in M}\|V\|_{L^{n/\alpha}(B(x,2\delta))}\Big)^p\leq 1/2
 \end{equation}
where $C_0=\max(C_{1},C_{2}) $, the constants appeared in Proposition \ref{p1} and \ref{p2}.

Now recall \eqref{b3} for any $u\in C^\infty(M)$
\begin{equation}
\begin{aligned}
   u(x)=&\Big(\tfrac{(\lambda+i)^{(2-\alpha)}}{\alpha/2}T_{\lambda,1}+ \tilde{T}_\lambda\Big)\Big((-\lap)^{\alpha/2}-z+V\Big)u-\Big(\tfrac{(\lambda+i)^{(2-\alpha)}}{\alpha/2}T_{\lambda,1}+ \tilde{T}_\lambda\Big)(Vu) \\
   &+\Big(\tfrac{(\lambda+i)^{(2-\alpha)}}{\alpha/2}R_{\lambda,1} +\tilde{R}_\lambda\Big)((-\lap)^{\alpha/2}-z)u.
   \end{aligned}
\end{equation}
By Propositions  \ref{p1}, \ref{p2} and \ref{p3}, and the local property of $T_{\lambda,1}$ and  $\tilde{T}_\lambda$, we can estimate the $L^{p}$ norms of each of the terms over one of our $\delta$ balls as follows:
\begin{equation}\label{d3}
\begin{aligned}
   \|u\|_{L^{p}(B_j)}^p\leq\Big(C\lambda^{1-\alpha+\sigma(p)}\|((-\lap)^{\alpha/2}-z+V)u\|_{L^2(M)} +& C_\delta \lambda^{\sigma(p)}\|u\|_{L^2(M)} \Big)^p \\
   +&\Big(C_0\|Vu\|_{L^{p^*}(B_j^*)} \Big)^p.
   \end{aligned}
\end{equation}
Since we have $\frac{1}{p^*}-\frac{1}{p}=\frac{\alpha}{n},$
by Holder's inequality,
$$\|Vu\|_{L^{p^*}(B_j^*)} \leq\|V\|_{L^{\frac n\alpha}(B_j^*)} \|u\|_{L^{p}(B_j^*)}.
$$
Since $M$ is the union of the $B_j$, and the number of these balls is $\approx \delta^{-n}$, if we add up the bound in \eqref{d3}
and use \eqref{d1} and \eqref{d2} we get
\begin{equation}
\begin{aligned}
   &\|u\|_{L^{p}(M)}^p\leq \sum\limits_j  \|u\|_{L^{p}(B_j)}^p\\
   &\leq C_\delta\Big(\lambda^{1-\alpha+\sigma(p)}\|((-\lap)^{\alpha/2}-z+V)u\|_{L^2(M)} +  \lambda^{\sigma(p)}\|u\|_{L^2(M)} \Big)^p \\
   &\hspace{72mm}+\Big(C_0\sup\limits_j\|V\|_{L^{\frac n\alpha}(B_j^*)} \Big)^p\sum\limits_j\|u\|_{L^{p}(B_j^*)}^p \\
    &\leq C_\delta\Big(\lambda^{1-\alpha+\sigma(p)}\|((-\lap)^{\alpha/2}-z+V)u\|_{L^2(M)} +  \lambda^{\sigma(p)}\|u\|_{L^2(M)} \Big)^p \\
   &\hspace{72mm}+C_M\Big(C_0\sup\limits_j\|V\|_{L^{\frac n\alpha}(B_j^*)} \Big)^p\|u\|_{L^{p}(M)}^p \\
    &\leq C_\delta\Big(\lambda^{1-\alpha+\sigma(p)}\|((-\lap)^{\alpha/2}-z+V)u\|_{L^2(M)} +  \lambda^{\sigma(p)}\|u\|_{L^2(M)} \Big)^p
 +\tfrac12\|u\|_{L^{p}(M)}^p
   \end{aligned}
\end{equation}
which implies
\begin{equation}
\|u\|_{L^{p}(M)}\leq C_{\delta}\lambda^{1-\alpha+\sigma(p)}\|((-\lap)^{\alpha/2}-z+V)u\|_{L^2(M)} +C_{\delta}  \lambda^{\sigma(p)}\|u\|_{L^2(M)}.
\end{equation}

\section{ Proof of Theorem \ref{quasi2} for $n\ge3$}
 In this section we first prove \eqref{quassi} when $n\ge3$ and $p_c\le p<\frac{2n}{n+1-2\alpha}$ by applying Propositions \ref{p1}, \ref{p2} and \ref{p3}. Next, we use heat kernel method and spectral theorem to deal with \eqref{quassi} for $n\ge3$ and $p\ge\frac{2n}{n+1-2\alpha}$. At the end, we prove \eqref{quasssi} for $n\ge4$ and $\frac{2n}{n-2\alpha}<p\le \infty$.

Since $u\in {\rm Dom}(H_V)$ may not be in $C^\infty(M)$, we need to take adjoint in \eqref{b2}
\begin{equation}\label{repro3}
        I=((-\lap)^{\alpha/2}-\bar{z})(\tfrac{(\lambda-i)^{(2-\alpha)}}{\alpha/2}T_{\lambda,1}^{*}+ \tilde{T}_\lambda^{*})+((-\lap)^{\alpha/2}-\bar{z})(\tfrac{(\lambda-i)^{(2-\alpha)}}{\alpha/2}R_{\lambda,1}^{*} +\tilde{R}_\lambda^{*})
\end{equation}
where $\bar{z}=(\lambda-i)^\alpha$, and $T^*$ is the adjoint of $T$.

It is not hard to check the image $e^{-H_V}[L^2]$ is an operator core for $H_V$ by spectral theorem and the heat kernel bounds in Proposition \ref{bounds_heat_kernel}. By the heat kernel bounds, we have ${L^\infty}\supset e^{-H_V}[L^2]$, so to prove \eqref{quassi} for $n\ge3$ and $p_c\le p<\frac{2n}{n+1-2\alpha}$,  it suffices to show that for $u\in {L^\infty}\cap \text{Dom}(H_V)$,
\begin{equation}\label{7.2}
\begin{aligned}
|\int u\hspace{0.5mm}\psi \hspace{0.5mm}dx| \le  C_V\lambda^{\sigma(p)}\big(\lambda^{1-\alpha}\|(\lapp+V-(\lambda+i)^\alpha)u\|_2 +\|u\|_2\big)+ \tfrac12\|u\|_{p} \qquad\quad\\
{\rm for}\hspace{2mm} \psi \in C^\infty(M)\hspace{2mm}  {\rm with}\hspace{2mm}  \|\psi\|_{p^\prime}=1.
\end{aligned}
\end{equation}
If we abbreviate the left side as $|(u,\psi)| $, then by \eqref{repro3} above
\begin{align*}
|(u,\psi)| &\le |(u, (\lapp-\bar{z}) \tfrac{(\lambda-i)^{(2-\alpha)}}{\alpha/2}T_{\lambda,1}^{*}\psi)|+|(u, (\lapp-\bar{z}) \tilde{T}_\lambda^{*}\psi)| \\
&\hspace{1in}+|(u, (\lapp-\bar{z}) \tfrac{(\lambda-i)^{(2-\alpha)}}{\alpha/2}R_{\lambda,1}^{*}\psi)|+|(u, (\lapp-\bar{z}) \tilde{R}_\lambda^{*}\psi)|\\
&\le  |\big( (\lapp+V-\bar{z}) u, \tfrac{(\lambda-i)^{(2-\alpha)}}{\alpha/2}T_{\lambda,1}^{*}\psi\big)|+|\big((\lapp+V-\bar{z})u,  \tilde{T}_\lambda^{*}\psi\big)| +|(Vu,\tilde{T}_\lambda^{*}\psi)|\\
&\hspace{0.2in} +|(Vu,\tfrac{(\lambda-i)^{(2-\alpha)}}{\alpha/2}T_{\lambda,1}^{*}\psi)|
+|(u, (\lapp-\bar{z}) \tfrac{(\lambda-i)^{(2-\alpha)}}{\alpha/2}R_{\lambda,1}^{*}\psi)|+|(u, (\lapp-\bar{z}) \tilde{R}_\lambda^{*}\psi)|.
\end{align*}
Now we need to estimate the six terms on the right.

First, by duality, \eqref{c1} yields $ \|\frac{(\lambda-i)^{(2-\alpha)}}{\alpha/2}T_{\lambda,1}^{*}\|_{L^{p^\prime} \rightarrow L^2}\le C_\delta\lambda^{1-\alpha+\sigma(p)}$, and so
 \begin{equation}\label{7.3}
\begin{aligned}
|\big( (\lapp+V-\bar{z}) u, \tfrac{(\lambda-i)^{(2-\alpha)}}{\alpha/2}T_{\lambda,1}^{*}\psi\big)|&\le  \|(\lapp+V-\bar{z}) u\|_2 \|\tfrac{(\lambda-i)^{(2-\alpha)}}{\alpha/2}T_{\lambda,1}^{*}\psi\|_2 \\
 &\le C_\delta \lambda^{1-\alpha+\sigma(p)}  \|(\lapp+V-\bar{z}) u\|_2.
\end{aligned}
\end{equation}
Similarly by \eqref{c7}  we have $\|(\lapp-\bar{z}) \frac{(\lambda-i)^{(2-\alpha)}}{\alpha/2}R_{\lambda,1}^{*}\|_{L^{p^\prime} \rightarrow L^2}\le C_\delta\lambda^{\sigma(p)}$, so
 \begin{equation}\label{7.4}
|(u,(\lapp-\bar{z}) \tfrac{(\lambda-i)^{(2-\alpha)}}{\alpha/2}R_{\lambda,1}^{*}\psi)|\le  \| u\|_2\|(\lapp-\bar{z}) \tfrac{(\lambda-i)^{(2-\alpha)}}{\alpha/2}R_{\lambda,1}^{*}\psi\|_2\le C_\delta \lambda^{\sigma(p)} \| u\|_2.
\end{equation}
The operator $\tilde{T}_\lambda^{*}$ and $\tilde{R}_\lambda^{*}$ can be dealt with in the same way by \eqref{c4} and \eqref{c8}.

Next, we only need to handle the remaining two terms $|(Vu,\frac{(\lambda-i)^{(2-\alpha)}}{\alpha/2}T_{\lambda,1}^{*}\psi)|$ and $ |(Vu,\tilde{T}_\lambda^{*}\psi)|$. We follow from the same argument from previous section to see that, if we choose
$\delta>0$ small enough, we can find a collection of $\delta$-balls $B_j$ so that if $B_j^*$ is the double then:
$ \sum\limits_{j=1}\limits^{N_\delta}\textbf{1}_{B_j^*}(x)\leq C_{M}$ and $C_{M} C_0\sup\limits_{x\in M}\|V\|_{L^{n/\alpha}(B(x,2\delta))}\leq \frac12$, where $C_0=\max(C_1,C_2)$ with $C_1$, $C_2$ in Propositions \ref{p1} and \ref{p2}.

And similar to $T_{\lambda,1}$, the adjoint operator $T_{\lambda,1}^{*}$ is a ``local operator'', again by duality, \eqref{c2} yields
    \begin{equation}\label{f2}
     \| \tfrac{(\lambda-i)^{(2-\alpha)}}{\alpha/2}T_{\lambda,1}^{*} f\|_{L^{r^\prime}(B_j)}\leq C_{1} \|f\|_{L^{p^\prime }(B_j^*)}
  \end{equation}
 where $\frac1r=\frac1p+\frac\alpha n$ and $\frac1{r^\prime}=1-\frac1r$. Consequently,
\begin{equation}\label{Tbound}
\begin{aligned}
|(Vu,\tfrac{(\lambda-i)^{(2-\alpha)}}{\alpha/2}T_{\lambda,1}^{*}\psi)|&\le \sum\limits_j|(\textbf{1}_{B_j}Vu,\textbf{1}_{B_j}\tfrac{(\lambda-i)^{(2-\alpha)}}{\alpha/2}T_{\lambda,1}^{*}\psi)| \\
&\le \sum\limits_j \|Vu\|_{L^r(B_j)}\|\tfrac{(\lambda-i)^{(2-\alpha)}}{\alpha/2}T_{\lambda,1}^{*}\psi\|_{L^{r^\prime}(B_j)} \\
&\le C_0\sum\limits_j \|V\|_{L^{n/\alpha}(B_j)}\|u\|_{L^{p}(B_j)}\|\psi\|_{L^{p^\prime}(B_j^*)} \\
&\le \tfrac{1}{2C_M}\big(\sum\limits_j \|u\|_{L^{p}(B_j)}^{p}\big)^{\tfrac{1}{p}}\big(\sum\limits_j\|\psi\|_{L^{p^\prime}(B_j^*)}^{p^\prime}\big)^ {\tfrac{1}{p^\prime}} \\
&\le \tfrac12 \|u\|_{L^{p}(M)}.
\end{aligned}
\end{equation}

The term $|(Vu,\tilde{T}_\lambda^{*}\psi)|$ can be dealt with exactly same way by using \eqref{c5}. Therefore, we have proved \eqref{quassi} for $n\ge3$ and $p_c\le p<\frac{2n}{n+1-2\alpha}$ with $(\lambda+i)^\alpha$ replaced by $(\lambda-i)^\alpha$, but this can be done by choosing $z=(\lambda-i)^\alpha$ in the very beginning.

\vspace{8mm}Next, we prove \eqref{quassi} for $n\ge3$ and $p\ge\frac{2n}{n+1-2\alpha}$.

Let $R_\lambda : L^2 \rightarrow L^2$ denote the spectral projection operator associated with $P_V$  corresponding to the interval $(2\lambda,\infty)$, i.e., $R_\lambda=\textbf{1}_{P_V>2\lambda}$, so that
$$ R_\lambda f=\sum\limits_{\lambda_j>2\lambda} \langle f, e_j\rangle e_j
$$
where $\{e_j\}$ is an orthonormal basis of the eigenfunctions of $H_V$. And let $L_\lambda=\textbf{1}_{P_V\le2\lambda}$ denote the projection onto frequencies $\le\lambda$ so that $I=L_\lambda+R_\lambda$ if $R_\lambda$ is as above.

Now we may use the special case $p=p_c$ of  \eqref{quassi} we proved just now and the heat kernel bounds  in Proposition \ref{bounds_heat_kernel} to prove the ``low frequency estimates''
\begin{equation}\label{low}
\|L_\lambda u\|_{L^{p}(M)}\ls \lambda^{1-\alpha+\sigma(p )}\|(\lapp+V-(\lambda+i)^\alpha)u\|_{L^2(M)},  \quad \text{if} \hspace{2mm}p>p_c.
\end{equation}
First, by Young's inequality and Proposition \ref{bounds_heat_kernel}, we have the following:
\begin{equation}\label{young}
 \|e^{-tH_V}\|_{L^p(M)\rightarrow L^q(M)} \ls t^{-\frac n\alpha(\frac 1p-\frac 1q)}, \quad \text{if}\hspace{2mm} 0<t\le1, \ \text{and} \hspace{2mm} 1\le p\le q\le \infty.
\end{equation}
If we fix $t=\lambda^{-\alpha}$, then by spectral theorem the operator $$\tilde{L}_\lambda f= e^{\lambda^{-\alpha}H_V}L_\lambda f=\sum\limits_{\lambda_j\le2\lambda} e^{\lambda^{-\alpha}\lambda_j^{\alpha}} \langle f, e_j\rangle e_j
$$ satisfies the bound $ \|\tilde{L}_\lambda\|_{L^2\rightarrow L^2}\le C$.
So by using \eqref{quassi} for $p=p_c$ and \eqref{young},  we have for $p>p_c$
\begin{equation}
\begin{aligned}
\|L_\lambda u\|_{L^p(M)}&=\|e^{-\lambda^{-\alpha}H_V}L_\lambda e^{\lambda^{-\alpha}H_V}L_\lambda u\|_{L^p(M)}\ls \lambda^{n(\frac{1}{p_c}-\frac 1p)}\|L_\lambda \tilde{L}_\lambda u\|_{L^{p_c}(M)} \\
&\ls \lambda^{1-\alpha+\sigma(p_c)+n(\frac{1}{p_c}-\frac 1p)}\|(\lapp+V-(\lambda+i)^\alpha)\tilde{L}_\lambda u\|_{L^2(M)}
\end{aligned}
\end{equation}
Since $\sigma(p)= \frac{1}{p_c}+n(\frac{1}{p_c}-\frac 1p) $ and
\begin{align*}
  \|(\lapp+V-(\lambda+i)^\alpha)\tilde{L}_\lambda u\|_{L^2(M)}&=\|\tilde{L}_\lambda (\lapp+V-(\lambda+i)^\alpha)u\|_{L^2(M)}\\
&\le C\|(\lapp+V-(\lambda+i)^\alpha) u\|_{L^2(M)}.
\end{align*}
This proves the claim \eqref{low}.

Now we need to prove the ``high frequency estimates'' for the operator $R_\lambda$. We consider $n=3$ and $n\ge4$ separately.  Let $\{\beta_j\}_{j\ge0}$ be a sequence of bump functions on $\mathbb{R}$ satisfying $\beta_0(x)+ \sum_{j=1}^{\infty}\beta_j(x)=1 $ where $\beta_j(x)=\beta_1(2^{1-j}x)$
with $\text{supp}(\beta_1)\subset\{|x|\in(\frac12,2) \}$, and  $\text{supp}(\beta_0)\subset\{|x|\in[0,1) \}$. For each $\beta_j$, if we use \eqref{young} with $t=2^{-j}$, we have for $2\le p\le \infty$
\begin{equation}\label{heee}
\|\beta_j(H_V)f\|_{L^p}\ls  2^{j\frac n\alpha(\frac 12-\frac 1p)}\|e^{2^{-j}H_V}\beta_j(H_V)f\|_{L^2}\approx 2^{j\frac n\alpha(\frac 12-\frac 1p)}\|\beta_j(H_V)f\|_{L^2}.
\end{equation}

When $n=3$, by using \eqref{heee} we get
\begin{equation}\label{high2}
\begin{aligned}
 \|R_\lambda u\|_{L^\infty(M)}&\le \sum_{2^{j}\ge \lambda^\alpha}^{\infty}\|R_\lambda \beta_j(H_V)u\|_{L^\infty(M)}  \\
 &\ls \sum_{2^{j}\ge \lambda^\alpha}^{\infty}2^{j(\frac {3}{2\alpha}-1)}\|H_VR_\lambda u\|_{L^2(M)}\\
 &\ls  \|H_VR_\lambda u\|_{L^2(M)}\\
 &\approx \|(\lapp+V-(\lambda+i)^\alpha)R_\lambda u\|_{L^2(M)}
\end{aligned}
\end{equation}
where we use $\alpha>\frac32$ in the third inequality. Note that in \eqref{low} the exponent $1-\alpha+\sigma(p)\ge0$ if $p\ge \frac{2n}{n+1-2\alpha}$. By interpolation, this combined with the low frequency part \eqref{low} proves \eqref{quassi} for $n=3$, $\frac32<\alpha<2$ and $\frac{2n}{n+1-2\alpha}\le p\le \infty$.

When $n\ge4$, similar to $n=3$, it suffices to prove the endpoint case
\begin{equation}\label{high}
\|R_\lambda u\|_{L^{p_\alpha}(M)}\ls \|H_VR_\lambda u\|_{L^2(M)} \quad {\rm if} \hspace{2mm} p_\alpha:=\tfrac{2n}{n-2\alpha}
\end{equation}
since $\|H_VR_\lambda u\|_{L^2(M)}\approx \|(\lapp+V-(\lambda+i)^\alpha)R_\lambda u\|_{L^2(M)}$, and the exponent $1-\alpha+\sigma(p)\ge0$ if $p\ge \frac{2n}{n+1-2\alpha}$.

\vspace{2mm}By adding a constant if necessary, we assume $H_V$ is invertible, so we can represent the resolvent as
\begin{equation}
\begin{aligned}
H_V^{-1}&=\int_0^\infty e^{-H_Vt}dt\\
&=\int_0^1 e^{-H_Vt}dt+\int_1^\infty e^{-H_Vt}dt \\
&=H_0+H_1
\end{aligned}
\end{equation}
The operator $H_1$ is a good error term. To see this, by functional calculus, $$H_1=e^{-H_V} H_V^{-1}.$$ Then by \eqref{heee} we have
 \begin{equation}
    \begin{aligned}
    \|H_1f\|_{L^{p_\alpha}(M)}&\le \sum\limits_{j=0}^\infty  \|\beta_j(H_V)H_1 f\|_{L^{p_\alpha}(M)} \\
    &\ls \sum\limits_{j=0}^\infty 2^{j\frac n\alpha(\frac 12-\frac 1{p_\alpha})}\|\beta_j(H_V)H_1 f\|_{L^{2}(M)}\\
    &\ls \sum\limits_{j=0}^\infty 2^{j\frac n\alpha(\frac 12-\frac 1{p_\alpha})}e^{-2^j}2^{-j}\|f\|_{L^{2}(M)}\\
    &\ls\|f\|_{L^{2}(M)}.
    \end{aligned}
    \nonumber
\end{equation}
For the operator $H_0$, by using Proposition \ref{bounds_heat_kernel}, we have
\begin{equation}
|H_0(x,y)|=|\int_0^1 e^{-H_Vt}(x,y)dt|\ls \int_0^1 q_\alpha(t,x,y)dt \ls d_g(x,y)^{\alpha-n}.
\end{equation}
Since $2< p_\alpha<\infty$ for $n\ge4$, by Hardy-Littlewood-Sobolev inequality,
\begin{equation}
\|H_0f\|_{L^{p_\alpha}(M)}\ls \|f\|_{L^{2}(M)}.
\end{equation}
Hence we proved the claim \eqref{high}. Therefore, the proof of \eqref{quassi} for $n\ge3$ is complete.

Finally, when $n\ge4$ and $\frac{2n}{n-2\alpha}< p\le \infty $, by \eqref{heee}
if $N>\frac {n}{2}$
\begin{equation}
\begin{aligned}
 \|R_\lambda u\|_{L^p(M)}&\le \sum_{2^{j}\ge \lambda^\alpha}^{\infty}\|R_\lambda \beta_j(H_V))u\|_{L^p(M)}  \\
 &\ls \sum_{2^{j}\ge \lambda^\alpha}^{\infty}2^{j(\frac n\alpha(\frac{1}{2}-\frac 1p)-\frac N\alpha)}\|(I+H_V)^{N/\alpha}R_\lambda u\|_{L^2(M)}\\
 &\ls  \lambda^{-N+ n(\frac12-\frac1p)} \|(I+H_V)^{N/\alpha}R_\lambda u\|_{L^2(M)}.
\end{aligned}
\end{equation}
So \eqref{quasssi} follows from \eqref{low} and \eqref{high2}.

\section{Proof of Theorem \ref{quasi2} for $n=2$}
In this section, we prove \eqref{quassi} for $n=2$.  This is a unique case since we do not have inequalities like \eqref{c2}, \eqref{c5} for $\alpha\ge\frac53$. As in the previous section, it suffices to prove \eqref{7.2} for $p\in[p_c,\infty]=[6,\infty]$. By interpolation, we only need to prove it at the endpoints. We will first prove the case $p=\infty$, and then prove the case $p=p_c=6$.

As before, \eqref{c1} and \eqref{c7} yield \eqref{7.3} and \eqref{7.4}, respectively for all $p\in [p_c,\infty]=[6,\infty]$, and similarly if $1<\alpha<2$, then \eqref{c4} and \eqref{c8} can be used for the terms that include $\tilde{T}_\lambda^{*}$ and $\tilde{R}_\lambda^{*}$. So again we only need
to handle the remaining two terms $|(Vu,\frac{(\lambda-i)^{(2-\alpha)}}{\alpha/2}T_{\lambda,1}^{*}\psi)|$ and $ |(Vu,\tilde{T}_\lambda^{*}\psi)|$.

\subsection{The case $p=\infty$}
This case follows from the kernel estimates of $\frac{(\lambda+i)^{(2-\alpha)}}{\alpha/2}T_{\lambda,1}$ and $\tilde T_\lambda$, and the Kato condition \eqref{kato}.  We claim that for $\lambda\ge \delta^{-2/\alpha}$
\begin{equation}\label{kernel01}
  |\tfrac{(\lambda+i)^{(2-\alpha)}}{\alpha/2}T_{\lambda,1}(x,y)|\le C d_g(x,y)^{\alpha-2}\textbf{1}_{d_g(x,y)\le\delta},
\end{equation}
and
\begin{equation}\label{kernel02}
  |\tilde T_\lambda(x,y)|\le C d_g(x,y)^{\alpha-2}\textbf{1}_{d_g(x,y)\le\delta}
\end{equation}
where the constants are independent of $\lambda$ and $\delta$.
The proof of the claim will be given in the last two subsections. As in the preceding section, we may assume $u\in {L^\infty}\cap \text{Dom}(H_V)$, so the Kato condition $V\in\mathcal{K}_\alpha(M)$ and \eqref{kernel01} ensures that
$$ \tfrac{(\lambda+i)^{(2-\alpha)}}{\alpha/2} T_{\lambda,1} (Vu)(x)=\int_M \tfrac{(\lambda+i)^{(2-\alpha)}}{\alpha/2}T_{\lambda,1}(x,y)V(y)u(y) dy
$$
is given by an absolutely convergent integral, as is $|(Vu,\frac{(\lambda-i)^{(2-\alpha)}}{\alpha/2}T_{\lambda,1}^{*}\psi)|$. Hence by Fubini's theorem
$$ |(Vu,\tfrac{(\lambda-i)^{(2-\alpha)}}{\alpha/2}T_{\lambda,1}^{*}\psi)|=|(\tfrac{(\lambda+i)^{(2-\alpha)}}{\alpha/2}T_{\lambda,1}(Vu),\psi)|\le \|\tfrac{(\lambda+i)^{(2-\alpha)}}{\alpha/2}T_{\lambda,1}(Vu)\|_\infty.
$$
By the Kato condition \eqref{kato}, we have $ \|\frac{(\lambda+i)^{(2-\alpha)}}{\alpha/2}T_{\lambda,1}(Vu)\|_\infty \le \frac12\|u\|_\infty$, if $\delta$ is small enough. And the term $ |(Vu,\tilde{T}_\lambda^{*}\psi)|$ is handled exactly the same way.
This proves \eqref{7.2} when $p=\infty$.

\subsection{The case $p=p_c=6$}
When $n=2$ and $\frac{2n}{n+1}=\frac43<\alpha<\frac53$, we see that \eqref{c2} and \eqref{c5} are still valid at $p=p_c$, which imply the desired bounds like \eqref{Tbound} for these two terms at $p=p_c$. So we only need to consider $\frac53\le\alpha<2$. We claim that when $\frac53\le \alpha<2$
\begin{equation}\label{norm01}
  \|\tfrac{(\lambda+i)^{(2-\alpha)}}{\alpha/2}T_{\lambda,1}\|_{L^{\frac2\alpha}\to L^6}\ls \lambda^{-\frac13}
\end{equation}
and
\begin{equation}\label{norm02}
  \|\tilde T_\lambda\|_{L^{\frac2\alpha}\to L^6}\ls \lambda^{-\frac13}.
\end{equation}
The proof of the claim will be given in the last two subsections. As in the argument above it is enough to bound $\|(\frac{(\lambda+i)^{(2-\alpha)}}{\alpha/2}T_{\lambda,1}(Vu)\|_6$ and $\|\tilde{T}_\lambda(Vu)\|_6$. By \eqref{norm01},
$$\|\tfrac{(\lambda+i)^{(2-\alpha)}}{\alpha/2}T_{\lambda,1}(Vu)\|_6\ls\lambda^{-1/3}\|Vu\|_{\frac2\alpha}\ls\lambda^{-1/3}\|V\|_{\frac2\alpha}\|u\|_\infty.
$$
Since we are assuming that $V\in L^{\frac2\alpha}(M) $ and we just proved that
$$\|u\|_\infty\ls \lambda^{1/2} \lambda^{1-\alpha}\|(\lapp+V-(\lambda+i)^\alpha)u\|_{L^2(M)},
$$
we conclude that
$$\|\tfrac{(\lambda+i)^{(2-\alpha)}}{\alpha/2}T_{\lambda,1}(Vu)\|_6\ls \lambda^{1/6} \lambda^{1-\alpha}\|(\lapp+V-(\lambda+i)^\alpha)u\|_{L^2(M)}.
$$
The term $\|\tilde{T}_\lambda(Vu)\|_6$ can be dealt with by the same argument. So we finish the proof.

\subsection{Proof of \eqref{kernel01} and \eqref{norm01}}
Recall in \eqref{daa}, we have
\begin{equation}
T_{\lambda,1} f=\sum\limits_{j=0}^{j_0} S_jf
\nonumber
\end{equation}
with $\lambda^{-1}2^{j_0}\approx \delta$. We have obtained the kernel estimate for $\frac{(\lambda+i)^{(2-\alpha)}}{\alpha/2}S_0$ in \eqref{S0kernel} that
\begin{equation}\label{S0ker}|\tfrac{(\lambda+i)^{(2-\alpha)}}{\alpha/2}S_0(x,y)| \le C d_g(x,y)^{\alpha-2}.
\end{equation}
By the finite propagation speed property, $\frac{(\lambda+i)^{(2-\alpha)}}{\alpha/2}S_0(x,y)$ is supported in $\{d_g(x,y)\le 2\lambda^{-1}\}$.

For the operator $S_j$, we can use the following kernel bound (see  \cite[(2.29)]{BSSY}) for $n=2$
\begin{equation}
|S_j(x,y)|\le
\begin{cases}
C\lambda^{-1/2}\varepsilon^{-1/2} \hspace{3mm} \text{if} \hspace{2mm} d_g(x,y)\le 2\varepsilon \\
0  \hspace{22mm} \text{if} \hspace{2mm} d_g(x,y)> 2\varepsilon
\end{cases}
\end{equation}
where $\varepsilon=\frac{2^j}{\lambda}$  is a dyadic number with $\lambda^{-1}\le\varepsilon\le 2\delta$. By the finite propagation speed property, each $S_j(x,y)$ is supported in $\{d_g(x,y)\le \delta\}$. So if we summing over the dyadic numbers, we have
\begin{equation}
 |\sum\limits_{j=1}^{j_0} S_j(x,y) | \le
\begin{cases}
C\lambda^{-1/2}d_g(x,y)^{-1/2} \hspace{6mm} \text{if} \hspace{2mm} \lambda^{-1}\le d_g(x,y)\le \delta, \\
C  \hspace{33mm} \text{if} \hspace{2mm} d_g(x,y)< \lambda^{-1}.
\end{cases}
\end{equation}
Then the kernel of the operator $\frac{(\lambda+i)^{(2-\alpha)}}{\alpha/2}\sum\limits_{j=0}^{j_0} S_j$ satisfies
\begin{equation}
 |\tfrac{(\lambda+i)^{(2-\alpha)}}{\alpha/2}\sum\limits_{j=1}^{j_0} S_j(x,y) | \le
\begin{cases}
C\lambda^{3/2-\alpha}d_g(x,y)^{-1/2} \hspace{6mm} \text{if} \hspace{2mm} \lambda^{-1}\le d_g(x,y)\le \delta,  \\
C\lambda^{2-\alpha} \hspace{27.5mm} \text{if} \hspace{2mm} d_g(x,y)< \lambda^{-1},
\end{cases}
\end{equation}
which is better than \eqref{kernel01}. So we complete the proof of \eqref{kernel01}.

Next, to prove the operator norm estimate \eqref{norm01} for $\frac53\le \alpha<2$, we need to consider two cases $\alpha=\frac53$, and $\frac53<\alpha<2$ separately. When $\alpha=\frac53$, we recall \eqref{decomp}
\[\tfrac{(\lambda+i)^{(2-\alpha)}}{\alpha/2}T_{\lambda,1} =\tfrac{(\lambda+i)^{-1/3}}{\alpha/2}((-\lap-(\lambda+i)^{2})^{-1}-R_{\lambda,1}).\]
Using the resolvent estimates by Frank-Schimmer \cite[Theorem 1]{frank}, we have
\[\|(-\lap-(\lambda+i)^{2})^{-1}\|_{L^{6/5}\to L^6}\ls \lambda^{-\frac23}.\]
Moreover, we may use spectral projection bounds \cite{sogge88} as in the proof of \eqref{c1} to get
\[\|R_{\lambda,1}\|_{L^{6/5}\to L^6}\ls \lambda^{-\frac23}.\]
Thus, we finish the proof of \eqref{norm01} for $\alpha=\frac53$.

When $\frac53<\alpha<2$, the kernel estimate in \eqref{kernel01} is not precise enough to get the desired operator bounds. We need to exploit oscillatory integrals as in the proof of Proposition \ref{p1}. First, since the kernel of $\frac{(\lambda+i)^{(2-\alpha)}}{\alpha/2}S_0$ satisfies \eqref{S0ker} and is supported in $\{d_g(x,y)\le 2\lambda^{-1}\}$, by Young's inequality its norm satisfies the desired bound \eqref{norm01}. Next, as in the proof of Proposition \ref{p1}, we may write $S_j=K_j+T_j$ for $j\ge 1$, where the kernel of $K_j$ is supported in $\{d_g(x,y)\le 4\eps\}$, $\eps=2^j/\lambda$, and satisfies
\[|K_j(x,y)|\ls \eps^{1-n}\lambda^{-1}=\eps^{-1}\lambda^{-1}=2^{-j},\]
and $T_j$ satisfies the operator bound in \eqref{c3}. By Young's inequality and the summation of a geometric series,
\[\|\tfrac{(\lambda+i)^{(2-\alpha)}}{\alpha/2}\sum K_j\|_{L^{\frac2\alpha}\to L^6}\ls \sum_{j\ge1} 2^{\frac{4-3\alpha}3j}\lambda^{-\frac13}=\lambda^{-\frac13}.\]
Thus $\frac{(\lambda+i)^{(2-\alpha)}}{\alpha/2}\sum K_j$ satisfies the desired bound \eqref{norm01} when $\alpha>\frac43$. Moreover, as in the proof of Proposition \ref{p1}, by duality argument and interpolation, it is not hard to get (see Figure \ref{fig2})
\[ \|T_j\|_{L^{\frac2\alpha}\to L^6}\ls \begin{cases}\lambda^{\alpha-\frac73}2^{\frac{5-3\alpha}4j},\ \ \ \ \tfrac53<\alpha<\tfrac{17}9\\
\lambda^{\alpha-\frac73}2^{-\frac16j},\ \ \ \ \ \tfrac{17}9\le \alpha<2.
\end{cases}\]
By the summation of a geometric series, $\frac{(\lambda+i)^{(2-\alpha)}}{\alpha/2}\sum T_j$ also satisfies the bound \eqref{norm01} when $\alpha>\frac53$. Hence the proof of \eqref{norm01} is finished.

\subsection{Proof of \eqref{kernel02} and \eqref{norm02}}
 We first obtain the kernel estimates for $\tilde T_\lambda$, and use them to prove the operator bound \eqref{norm02}. We write
\begin{equation}
\begin{aligned}
        \tilde{T}_\lambda &=\frac{\sin(\pi \alpha/2)}{-\pi\sqrt{\gamma}}\int_0^\infty\int_0^\infty \frac{\gamma^{\alpha/2}\rho(t/\delta)e^{-\sqrt{\gamma}t}\cos tP }{\gamma^\alpha-2z\gamma^{\alpha/2}\cos(\pi \alpha/2)+z^2}d\gamma dt   \qquad \qquad \qquad \qquad \\
        &=\frac{\sin(\pi \alpha/2)}{\pi}\int_0^\infty\frac{\gamma^{\alpha/2} T_{0,\sqrt{\gamma}}}{\gamma^\alpha-2z\gamma^{\alpha/2}\cos(\pi \alpha/2)+z^2}d\gamma=\int_0^{\delta^{-2}}+\int_{\delta^{-2}}^\infty\\
        &:=T_1+T_2.
\end{aligned}
\end{equation}
By the finite propagation speed property, the kernels of $T_1$ and $T_2$ are supported in $\{d_g(x,y)\le \delta\}$.
Now we estimate the kernels of $T_1$ and $T_2$ separately.

\noindent \textbf{Case 1:} $ 0<\gamma\le \delta^{-2}$.

In this case, we first consider the multiplier associated to the operator $ T_{0,\sqrt{\gamma}}(P)$. By integration by parts (similar to \cite[Lemma 2.2]{sy}), it is not hard to get
\begin{equation}
|\tfrac{d^j}{d\tau^j}T_{0,\sqrt{\gamma}}(\tau)|\le C_j\delta^{-1}\gamma^{-1/2}(1+|\tau|)^{-2-j}.
\end{equation}
So the multiplier associated to the operator $T_1(P)$ satisfies
\begin{equation}\label{8.10}
|\tfrac{d^j}{d\tau^j}T_1(\tau)|\le C\int_0^{\delta^{-2}}\frac{\gamma^{\alpha/2} |\frac{d^j}{d\tau^j}T_{0,\sqrt{\gamma}}(\tau)|}{\gamma^\alpha+|z|^2}d\gamma \le
C\lambda^{-2\alpha}\delta^{-2-\alpha} (1+|\tau|)^{-2-j}.
\end{equation}
Since $-n$ order pesudodifferential operator in $\mathbb{R}^n$ has kernel bounded by $\log |x-y|^{-1}$ (see e.g. \cite{Taylor}), this implies
\[|T_1(x,y)|\le C \lambda^{-2\alpha}\delta^{-2-\alpha}|\log d_g(x,y)|\]
which is better than \eqref{kernel02} when $\lambda\ge\delta^{-2/\alpha}$.

\noindent \textbf{Case 2:} $ \gamma>\delta^{-2}$.

In this case, we write
 $T_{0,\sqrt{\gamma}} f=\sum\limits_{j=0}^{j_0} S_j^\gamma f$, where
 \begin{equation}
    S_0^\gamma f=\frac{-1}{\sqrt{\gamma}}\int_0^\infty\tilde{\rho}(\sqrt{\gamma} t)\rho(t/\delta)e^{-\sqrt{\gamma}t}\cos tP fdt
\end{equation}
 \begin{equation}
    S_j^\gamma f=\frac{-1}{\sqrt{\gamma}}\int_0^\infty\beta(\sqrt{\gamma}2^{-j}t)\rho(t/\delta)e^{-\sqrt{\gamma}t}\cos tP fdt,\ \ j=1,2,...
\end{equation}
By a simple integration by parts argument, it is not hard to prove that the multiplier associated to the operator $S_0^\gamma(P)$ satisfies
\begin{equation}\label{8.40}
|\tfrac{d^j}{d\tau^j}S_0^\gamma(\tau)|\le\begin{cases}
C_j|\tau|^{-2-j} \hspace{3mm} \text{if} \hspace{2mm}  |\tau|>2\sqrt{\gamma} \\
C_j\gamma^{-1-\frac j2} \hspace{5mm} \text{if} \hspace{2mm}  |\tau|\le2\sqrt{\gamma}\ .
\end{cases}
\end{equation}
So the multiplier associated to  $T_2^1(P):=\frac{\sin(\pi \alpha/2)}{\pi}\int_{\delta^{-2}}^\infty\frac{\gamma^{\alpha/2} S_0^\gamma(P)}{\gamma^\alpha-2z\gamma^{\alpha/2}\cos(\pi \alpha/2)+z^2}d\gamma  $ satisfies
\begin{equation}\label{8.14}
|\tfrac{d^j}{d\tau^j}T_2^1(\tau)|\le C\int_{\delta^{-2}}^\infty\frac{\gamma^{\alpha/2}| \tfrac{d^j}{d\tau^j}S_0^\gamma(\tau)|}{\gamma^\alpha+|z|^2}d\gamma \le \begin{cases}
C_j |\tau|^{-\alpha-j} \hspace{21mm} \text{if} \hspace{2mm}  |\tau|>2\lambda \\
C_j\lambda^{-\alpha}(1+|\tau|)^{-j}  \hspace{10mm} \text{if} \hspace{2mm}  |\tau|\le2\lambda
\end{cases}
\end{equation}
which implies
\[|T_2^1(x,y)|\le \begin{cases}
  C\lambda^{-\alpha}d_g(x,y)^{-2},\ \ \ d_g(x,y)>\lambda^{-1}\\
  Cd_g(x,y)^{\alpha-2},\ \ \ \ \ \ \ d_g(x,y)\le \lambda^{-1}.
\end{cases}\]
It is better than \eqref{kernel02}.

Now we only need to deal with $T_2^2(P):=\sum\limits_{j=1}^{j_0}\frac{\sin(\pi \alpha/2)}{\pi}\int_{\delta^{-2}}^\infty\frac{\gamma^{\alpha/2} S_j^\gamma(P)}{\gamma^\alpha-2z\gamma^{\alpha/2}\cos(\pi \alpha/2)+z^2}d\gamma$. If we follow the main strategy in \cite[Proof of (2.29)]{BSSY} and use \cite[Proposition 2.3]{sy} to refine the oscillatory integral estimates in \cite[Proposition 2.4]{BSSY}, we have
\begin{equation}
|S_j^\gamma(x,y)|\le
\begin{cases}
C\gamma^{-\frac12}\eps^{-1} \hspace{8mm} \text{if} \hspace{2mm} d_g(x,y)\le \varepsilon \\
0  \hspace{20.7mm} \text{if} \hspace{2mm} d_g(x,y)> \varepsilon
\end{cases}
\end{equation}
where $\varepsilon=\frac{2^j}{\sqrt{\gamma}}$  is a dyadic number with $\gamma^{-\frac12}\le\varepsilon\le \delta$. So if we summing over the dyadic numbers $\text{max}\big(\gamma^{-\frac12}, d_g(x,y)\big)\le\varepsilon\le \delta$, we have
\begin{equation}
 |\sum\limits_{j=1}^{j_0} S_j^\gamma(x,y) | \le
\begin{cases}
C\gamma^{-\frac12}d_g(x,y)^{-1} \hspace{11mm} \text{if} \hspace{2mm} \gamma^{-1/2}\le d_g(x,y)\le \delta, \\
C  \hspace{33mm} \text{if} \hspace{2mm} d_g(x,y)< \gamma^{-1/2}.
\end{cases}
\end{equation}
Hence the kernel  $T_2^2(x,y)=\sum\limits_{j=1}^{j_0}\frac{\sin(\pi \alpha/2)}{\pi}\int_{\delta^{-2}}^\infty\frac{\gamma^{\alpha/2} S_j^\gamma(x,y)}{\gamma^\alpha-2z\gamma^{\alpha/2}\cos(\pi \alpha/2)+z^2}d\gamma $ satisfies for $\alpha>1$
\begin{equation}\label{8.17}
|T_2^2(x,y)|\le C\int_{\delta^{-2}}^\infty\frac{\gamma^{\alpha/2}}{\gamma^\alpha+|z|^2}|\sum\limits_{j=1}^{j_0} S_j^\gamma(x,y)|d\gamma \le \begin{cases}
C\lambda^{1-\alpha} d_g(x,y)^{-1}  \hspace{10mm} \text{if} \hspace{2mm}  \lambda^{-1}\le d_g(x,y)\le \delta\\
C d_g(x,y)^{\alpha-2}  \hspace{16mm} \text{if} \hspace{2mm}  d_g(x,y)\le \lambda^{-1}
\end{cases}
\end{equation}
which is better than \eqref{kernel02}. So \eqref{kernel02} is proved. Consequently, the operator bound \eqref{norm02} follows immediately from three kernel estimates above and Young's inequality. In particular, we need  $\alpha>\frac43$ when using Young's inequality to estimate $\|T_2^2\|_{L^\frac2\alpha\to L^6}$.

\section{Strichartz estimates for the fractional wave equation}
In this section we will use the spectral projection estimates to prove Strichartz estimates for $H_V=\lapp+V$. As above, without loss of
generality, we shall assume that $H_V\ge0$. The following theorem generalizes the result for $(-\Delta_g)^{\alpha/2}$ by Dinh \cite[Corollary 1.4]{dinh}, and also generalizes the result for $-\Delta_g+V$ by Blair-Sire-Sogge \cite[Theorem 8.1]{BSS19}.
\begin{theorem}
Let $(M,g)$ be a compact Riemannian manifold of dimension $n\ge 2$. Suppose that $\frac{2n}{n+1}<\alpha<2$ and $V\in L^{n/\alpha}(M)\cap \mathcal{K}_\alpha(M)$. Let $u$ be a solution of
\begin{equation}
\begin{cases}
\big(\partial_t^2+\lapp+V\big)u=0  \\
u\big|_{t=0}=f_0, \hspace{2mm}  \partial u\big|_{t=0}=f_1
\end{cases}
\end{equation}
Then
\begin{equation}\label{stri}
\|u\|_{L^{\frac{2(n+1)}{n-1}}([0,1]\times M)} \ls \|(I+Q_V)^\gamma f_0\|_{L^2(M)} +\|(I+Q_V)^{\gamma-1} f_1\|_{L^2(M)}
\end{equation}
with $Q_V$ denoting $\sqrt{H_V}$, and $\gamma=\frac1\alpha+\frac1{p_c}(\frac2\alpha-1)$.
\end{theorem}
\begin{remark} {\rm As in Remark \ref{rem1}, the range of $\alpha$ here can be improved if the potential $V$ has better regularity. For example, \eqref{stri} holds for $0<\alpha<2$ if $V\equiv0$, and $1\le \alpha<2$ if $V\in L^\infty(M)$. The exponent $\gamma$ does not depend on $V$, and agrees with the one in \cite{dinh}. In \cite[Theorem 2.1]{BLP}, the authors show that the Strichartz estimates follows from the spectral projection bounds, and their proof works equally well in our circumstances. See also \cite{Nic}, \cite{BSS19}.}
\end{remark}
\begin{proof}
If, as above, $p_c=\frac{2(n+1)}{n-1}$, then to prove \eqref{stri}  it suffices to show that
\begin{equation}
\|e^{itQ_V}f\|_{L^{p_c}([0,1]\times M)} \ls \|(I+Q_V)^\gamma f\|_{L^2(M)}    \quad {\rm if}\hspace{2mm} \gamma=\tfrac1\alpha+\tfrac1{p_c}(\tfrac2\alpha-1).
\end{equation}

To prove this, it suffices to prove that whenever we fix $\rho\in \mathcal{S} (\mathbb{R})$ satisfying supp $\hat{\rho}\subset (-\frac12,\frac12)$ we have
\begin{equation}\label{local3}
\|\rho(t)e^{itQ_V}f\|_{L^{p_c}(\mathbb{R}\times M)} \ls \|(I+Q_V)^\gamma f\|_{L^2(M)}.
\end{equation}

Let
$$ \chi_{k,\alpha}f:=\sum\limits_{\lambda_j^{\alpha/2}\in(k,k+1]}E_jf, \quad E_jf=\langle f,e_{\lambda_j}\rangle e_{\lambda_j}
$$
so that $f=\sum_{k=0}^\infty \chi_{k,\alpha}f $. We split \[\chi_{k,\alpha}f=\sum_{k^{2/\alpha}<\lambda\le (k+1)^{2/\alpha}}\chi_\lambda^Vf\]
and there are $\ls (1+k)^{\frac2\alpha-1}$ terms.
By Corollary \ref{coro1}
\begin{equation}\label{spec}
\| \chi_\lambda^Vf\|_{L^{p_c}(M)} \ls (1+\lambda)^{1/p_c} \|f\|_{L^2(M)}
\end{equation}
Since $p_c>2$, by using Cauchy-Schwartz inequality, we have
\begin{equation}\label{specc}
\begin{aligned}
\| \chi_{k,\alpha}f\|_{L^{p_c}(M)}&\ls (1+k)^{\frac1\alpha-\frac12}  \|\big( \sum|\chi_\lambda^Vf|^2\big)^{\frac12}\|_{L^{p_c}(M)} \\
&\ls (1+k)^{\frac1\alpha-\frac12} \big(\sum  \| \chi_\lambda^Vf\|_{L^{p_c}(M)}^2\big)^{\frac12} \\
&\ls (1+k)^{\frac1\alpha-\frac12} \lambda^{\frac1{p_c}} \big(\sum  \| \chi_\lambda^Vf\|_{L^{2}(M)}^2\big)^{\frac12} \\
 &\le (1+k)^{\frac1\alpha-\frac12+\frac{2}{\alpha p_c}} \|f\|_{L^2(M)}
\end{aligned}
\end{equation}

To use \eqref{specc},  we first note that by Sobolev estimates

\begin{equation}
\|\rho(t)e^{itQ_V}f\|_{L^{p_c}(\mathbb{R}\times M)} \ls \||D_t|^{1/2-1/p_c}\big(\rho(t)e^{itQ_V} f\big)\|_{L_x^{p_c}L_t^2(\mathbb{R}\times M)}
\end{equation}
\newline
\noindent  If we let
$$ F(t,x)=|D_t|^{1/2-1/p_c}\big(\rho(t)e^{itQ_V} f(x)\big)
$$
denote the function inside the mixed-norm in the right, then
$$ F(t,x)=\sum\limits_{k=0}\limits^\infty F_k(t,x)
$$
where
$$ F_k(t,x)=|D_t|^{1/2-1/p_c}\big(\rho(t)e^{itQ_V} \chi_{k,\alpha}f(x)\big).
$$
Therefore, its $t$-Fourier transform is
\begin{equation}\label{8.6}
\hat{F_k}(\tau,x)=|\tau|^{1/2-1/p_c}\sum\limits_{\lambda_j^{\alpha/2}\in(k,k+1]}\hat{\rho}(\tau-\lambda_j^{\alpha/2})E_jf(x).
\end{equation}
Since we are assuming supp$\hat{\rho}\subset(-\frac12,\frac12)$, we get
$$\int_{-\infty}^\infty F_k(t,x)\overline{F_\ell(t,x)} dt=(2\pi)^{-1}\int_{-\infty}^\infty \hat{F_k}(\tau,x)\overline{\hat{F_\ell}(\tau,x)} d\tau=0
\hspace{4mm} \text{when} \hspace{2mm}|k-\ell|>10.
$$
Then
\begin{equation}
\begin{aligned}
\big(\int_{-\infty}^\infty \big| |D_t|^{1/2-1/p_c}&\big(\rho(t)e^{itQ_V} \chi_{k,\alpha}f(x)\big)\big|^2 dt\big)^{1/2} \\
&\ls \big(\int_{-\infty}^\infty \sum\limits_{k=0}\limits^\infty\big| F_k(t,x)\big|^2 dt\big)^{1/2} =(2\pi)^{-1}\big(\int_{-\infty}^\infty \sum\limits_{k=0}\limits^\infty\big| \hat{F_k}(\tau,x)\big|^2 d\tau\big)^{1/2}.
\nonumber
\end{aligned}
\end{equation}
Since $p_c>2$, we conclude that this implies that the square of the left side of \eqref{local3} is dominated by
$$\sum\limits_{k=0}\limits^\infty \int_{-\infty}^\infty \| \hat{F_k}(\tau,x)\|_{L^{p_c}(M)}^2 d\tau.
$$
Recalling \eqref{8.6}, the support properties of $\hat{\rho}$, we see that this along with \eqref{specc} and orthogonality imply that
the square of the  left side of \eqref{local3} is dominated by
\begin{equation}
\begin{aligned}
\sum\limits_{k=0}\limits^\infty \int_{-\infty}^\infty |\tau|^{1-2/p_c}\| &\sum\limits_{\lambda_j\in(k,k+1]}\hat{\rho}(\tau-\lambda_j)E_jf\|_{L^{p_c}(M)}^2 d\tau \\
&=\sum\limits_{k=0}\limits^\infty \int_{k-10}^{k+10} |\tau|^{1-2/p_c}\Big\| \sum\limits_{\lambda_j^{\alpha/2}\in(k,k+1]}\hat{\rho}(\tau-\lambda_j^{\alpha/2})E_jf\Big\|_{L^{p_c}(M)}^2 d\tau \\
&\ls \sum\limits_{k=0}\limits^\infty (1+k)^{1-2/p_c}(1+k)^{\frac2\alpha-1+\frac{4}{\alpha p_c}} \|\chi_{k,\alpha}f\|_{L^2(M)}^2 \\
&= \sum\limits_{k=0}\limits^\infty \|(1+k)^\gamma \chi_{k,\alpha}f\|_{L^2(M)}^2 =\|(I+Q_V)^\gamma f\|_{L^2(M)}^2
\end{aligned}
\end{equation}
as desired, which completes the proof.
\end{proof}

\bibliographystyle{plain}

\begin{thebibliography}{0}
\bibitem{BLP} Nicolas Burq, Gilles Lebeau, and Fabrice Planchon, Global existence for energy critical waves in 3-D
domains, J. Amer. Math. Soc. 21 (2008), no. 3, 831–845.
\bibitem{BSS19} M. Blair, Y. Sire, C. Sogge, Quasimode, eigenfunction and spectral projection bounds for Schr\"odinger operators on manifolds with critically singular potentials, J Geom Anal (2019). https://doi.org/10.1007/s12220-019-00287-z
\bibitem{BSS} K. Bogdan, A. St\'os and P. Sztonyk, Harnack inequality for stable processes on d-sets, Studia Math. 158 (2003), 163-198.
\bibitem{bj07} K. Bogdan and T. Jakubowski, Estimates of heat kernel of fractional Laplacian pertubed by gradient operators, Commun. Math. Phys. 271 (2007), 179-198.
\bibitem{BSSY} J. Bourgain, P. Shao, C.D. Sogge, and X. Yao, On $L^p$-resolvent estimates and the density of eigenvalues for compact Riemannian manifolds, Communications in Mathematical Physics 333.3 (2015): 1483-1527.
\bibitem{carmona} R. Carmona, W. Masters and  B. Simon, Relativistic Schr\"dinger operators: asymptotic behavior of the eigenfunctions. J. Funct. Anal. 91 (1990), no. 1, 117-142.
\bibitem{MSbook} Celso Martinez Carracedo and Miguel Sanz Alix, The theory of fractional powers of operators, volume
187 of North-Holland Mathematics Studies. North-Holland Publishing Co., Amsterdam, 2001.
\bibitem{cks} Z.-Q. Chen, P. Kim, and R. Song, Stability of Dirichlet heat kernel estimates for non-local operators under Feynman-Kac perturbation, Transactions of the American Mathematical Society 367.7 (2015): 5237-5270.
\bibitem{chensong} Z.-Q. Chen and R. Song, Conditional gauge theorem for non-local Feynman-Kac transforms, Probab. Th. rel. Fields,125(2003), 45-–72.
\bibitem{daubechies} I. Daubechies and E. H. Lieb, One-electron relativistic molecules with Coulomb interaction. Comm. Math. Phys. 90 (1983), no. 4, 497-510.
\bibitem{dinh} V. D. Dinh, Strichartz estimates for the fractional Schrödinger and wave equations on compact manifolds without boundary, J. of Diff. Eq., 12 (2017), 8804-8837
\bibitem{FKS}D. Ferreira, C.E. Kenig and M. Salo, On $L^p$ resolvent estimates for Laplace-–Beltrami operators on compact manifolds, Forum Mathematicum. Vol. 26. No. 3. De Gruyter, 2014.
\bibitem{frank} R.L. Frank and L. Schimmer, Endpoint resolvent estimates for compact Riemannian manifolds. Journal of Functional Analysis 272.9 (2017): 3904-3918.
\bibitem{FLS1} R. L. Frank, E. H. Lieb and R. Seiringer, Stability of relativistic matter with magnetic fields for nuclear charges up to the critical value. Comm. Math. Phys. 275 (2007), no. 2,479-489.
\bibitem{FLS2} R. L. Frank, E. H. Lieb and R. Seiringer,  Hardy-Lieb-Thirring inequalities for fractional Schr\"odinger operators. J. Amer. Math. Soc. 21 (2008), no. 4, 925-950.
\bibitem{GG} H. Gimperlein and G. Grubb, Heat kernel estimates for pseudodifferential operators, fractional Laplacians and Dirichlet-to-Neumann operators, J. Evol. Equ. 14 (2014), 49-83.
\bibitem{HSZ} X. Huang, Y. Sire and C. Zhang, Interior estimates for the eigenfunctions of the fractional Laplacian on a bounded Euclidean domain, arXiv:1907.08107.
\bibitem{KRS} C. E. Kenig, A. Ruiz, and C. D. Sogge. Uniform Sobolev inequalities and unique continuation
for second order constant coefficient differential operators. Duke Math. J., 55:329–347, 1987.
\bibitem{Laskin}N. Laskin, Fractional quantum mechanics and L\'evy path integrals, Phys. Lett. A, 268 (2000), 298-305.
\bibitem{Laskin2}N. Laskin, Fractals and quantum mechanics, Chaos: An Interdisciplinary Journal of Nonlinear Science 10.4 (2000): 780-790.
\bibitem{Laskin3}N. Laskin, Fractional Schr\"odinger equation, Physical Review E 66.5 (2002): 056108.
\bibitem{LY} P. Li and S.T. Yau, On the parabolic kernel of the Schr\"odinger operator, Acta Math. 156 (1986), 153-201.
\bibitem{LiebYau1} E. H. Lieb and  H.-T. Yau,  The Chandrasekhar theory of stellar collapse as the limit of quantum mechanics. Comm. Math. Phys. 112 (1987), no. 1, 147-174.

\bibitem{LiebYau2} E. H. Lieb and H.-T. Yau, The stability and instability of relativistic matter. Comm. Math. Phys. 118 (1988), no. 2, 177-213.

\bibitem{Nic} F. Nicola, Slicing surfaces and the Fourier restriction conjecture, Proc. Edinb. Math. Soc. (2) 52
(2009), no. 2, 515–527.



\bibitem{sal} L. Saloff-Coste, A note on Poincar\'e, Sobolev, and Harnack inequalities, IMRN. 1992, No. 2

\bibitem{sy} P. Shao and X. Yao, Uniform Sobolev resolvent estimates for the Laplace-Beltrami operator on compact manifolds, IMRN 2014.12 (2014): 3439-3463
\bibitem{sogge88} C. Sogge, Concerning the $L^p$ norm of spectral clusters for second-order elliptic operators on compact manifolds. J. Funct. Anal., 77(1):123-138, 1988.
\bibitem{fio} C. Sogge, Fourier integrals in classical analysis, Cambridge University Press, 1993.
\bibitem{song2} R. Song, Two-sided estimates on the density of the Feynman-Kac semigroups of stable-like processes, Electronic Journal of Probability, Vol. 11 (2006), Paper no. 6, pages 146-161.

\bibitem{stein} E.M. Stein, Harmonic Analysis: real-variable methods, orthogonality, and osciollatory integrals, Princeton University Press, Princeton, NJ, 1993.
\bibitem{sturm} K.T. Sturm, Heat kernel bounds on manifolds, Math. Ann. 292(1992), 149-–162.

\bibitem{Taylor} M. Taylor, Pesudodifferential operators and nonlinear PDE. Progress in Mathematics, 100. Birkh\"auser Boston, Inc., Boston, MA, 1991.
\bibitem{wz15} F. Wang and X.C. Zhang, Heat kernel for fractional diffusion operatoers with pertubations, Forum Math. 27 (2015), 973-994
  \end{thebibliography}

\end{document}